\documentclass[english]{amsart}

\usepackage{amssymb,amsmath,amsthm,esint,braket,mathtools,enumitem,hyperref}
\usepackage[capitalise]{cleveref}
\usepackage[T1]{fontenc}
\usepackage{microtype}

\providecommand{\abs}[1]{\left\vert #1 \right\vert}
\providecommand{\norm}[1]{\left\Vert #1 \right\Vert}

\providecommand{\pt}[1]{\left( #1 \right)}

\newcommand{\RR}{\mathbb R}
\newcommand{\NN}{\mathbb N}

\newcommand{\ve}{\varepsilon}

\newtheorem{theorem}{Theorem}[section]
\newtheorem{proposition}{Proposition}[section]
\newtheorem{lemma}{Lemma}[section]

\newtheorem{remark}{Remark}[section]

\Crefname{corollary}{Corollary}{Corollaries}
\Crefname{lemma}{Lemma}{Lemmas}
\Crefname{theorem}{Theorem}{Theorems}
\Crefname{proposition}{Proposition}{Propositions}

\begin{document}
	
\title{\(L^p\) theory for a singular Sturm-Liouville equation}

\author{Hernán Castro}
\email{hcastro@utalca.cl}
\address{Instituto de Matem{\'a}ticas, Universidad de Talca, Casilla 747, Talca, Chile}

\author{Iván Proaño}
\email{ivan.proano@utalca.cl}
\address{Instituto de Matem{\'a}ticas, Universidad de Talca, Casilla 747, Talca, Chile}
\thanks{I.P. was supported by \textit{ANID} becas Doctorado Nacional 21240838}

\date{\today}

\begin{abstract}
	In this paper we consider the following Sturm-Liouville equation
	\begin{equation}\tag{\(\star\)} \label{abs-eq}
		\left\{
		\begin{aligned}
			-(x^{2\alpha}u'(x))'+u(x)&=f(x)  && \text{ on } (0,1],\\
			u(1)&= 0, &&
		\end{aligned}
		\right. 
	\end{equation} 
	where $\alpha<1$ is a nonzero real number and $f$ belongs to $L^p(0,1)$ for \(p\geq 1\). We analyze the existence and regularity of solutions to \eqref{abs-eq} under suitable weighted Dirichlet boundary condition at the origin.
\end{abstract}

\date{\today}

\maketitle

\section{Introduction}

In this work we are interested in the following family of Sturm-Liouville equations
\begin{equation}\label{SL-eq}
	\left\{
	\begin{aligned}
		-(x^{2\alpha}u'(x))'+u(x)&=f(x)  && \text{ on } (0,1],\\
		u(1)&= 0, &&
	\end{aligned}
	\right.
\end{equation}
where \(f\in L^p(0,1)\) for \(1\leq p\leq \infty\) and \(\alpha<1\) is a non-zero real number.

The literature is vast regarding both regular and singular Sturm-Liouville equations and it is not our intention to go through the history of these equations, but the interested reader might want to use the monograph of Anton Zettl \cite{Zet2005} and the comprehensive reference list therein as a starting point into the theory.

It is also worth mentioning that the choice of the weight \(x^{2\alpha}\) is not arbitrary. As it was shown by Stuart \cite[Section 1.1]{St01}, if one considers an unshearable, inextensible rod whose resistance to bending is governed by the Bernoulli–Euler law then the differential operator
\[
L_Au(x):=-(A(x)u'(x))'
\]
appears naturally. In that model established by Stuart, the function \(A\) represents the profile of the tapered rod, and it is said to be of order \(p\) if
\[
\lim_{x\to 0}\frac{A(x)}{x^{p}}=L
\]
for some positive constant \(L\). A study of the spectrum of \(L_A\) and other relevant results regarding a non-linear problem modeling the tapered rod were established by Stuart and Vuillaume. We refer the interested to the series of articles \cite{St00-1,St01,St02,SV03,SV04,Vui03,V03}.

Later, Castro and Wang in \cite{CW11-1, CW11-2, Cas2017-1} studied \eqref{SL-eq} for the case \(\alpha>0\) and \(f\in L^2\). In those articles the problem of existence and uniqueness was studied together a study of the spectral properties of the operator \(L_\alpha u(x) =-(x^{2\alpha}u'(x))'+u(x)\), but the question of regularity and behavior of solutions near the origin was of particular interest.

On the one hand, it was shown in \cite{CW11-1} that due to the weight \(x^{2\alpha}\) the pure Dirichlet condition \(u(0)=0\) only makes sense for \(\alpha<\frac12\) and in that case the unique solution \(u_D\) satisfies
\[
x^{2\alpha-1}u_D\in H^1(0,1)
\]
so that in particular \(u_D(x)\) not only vanishes at \(x=0\) but it vanishes like \(x^{1-2\alpha}\).

On the other hand, a weighted Neumann boundary condition \(x^{2\alpha}u'(x)\big|_{x=0}=0\) can be imposed for every \(\alpha>0\). Just as in the Dirichlet case, the behavior near the origin of the unique solution \(u_N\) under this condition is better than the natural one in the sense that one verifies that
\[
x^{2\alpha-1}u'\in L^2(0,1)
\quad\text{and}\quad
\lim_{x\to 0^+}x^{2\alpha-\frac12}u'(x)=0.
\]

The main motivation of this article is to extend what was done in \cite{CW11-1} to the case \(f\in L^p(0,1)\) for \(1\leq p\leq \infty\) while also considering the case \(\alpha<0\) which was not studied before. As the reader can see in the \cref{app-bessel}, solutions to equation \eqref{SL-eq} can be represented rather explicitly with the aid of Bessel functions, however one of the purposes of this article is to avoid explicit representation of solutions and instead use tools that might be used for more general problems. For instance, as it was mentioned at the beginning, it would be of interest to study \eqref{SL-eq} where instead of the function \(x^{2\alpha}\) we have a general weight satisfying \(A(x)\sim x^{2\alpha}\) near the origin in suitable fashion: in that case one might not have good representation formulas for the solutions. Instead of going through that road we use tools that are common in the study of partial differential equations such as functional analysis methods and the use of \emph{a priori} bounds to prove existence and regularity of solutions. Additionally, while some of the techniques we use throughout this work make use of the one dimensional nature of the problem, there are also parts that can be generalized to similar problems in higher dimension such as
\begin{equation}\label{PDE-ex}
	-\mathrm{div}(A(x)\cdot\nabla u(x))+u(x)=f(x)  \qquad \text{ in } \Omega\subseteq\RR^N,\\
\end{equation}
where \(A(x)\) could be a matrix valued function with eigenvalues behaving like the radial weight \(\abs{x}^{\alpha}\) or the monomial weight \(x^A=\abs{x_1}^{a_1}\cdot\ldots\cdot\abs{x_N}^{a_N}\). In particular the use of weighted Sobolev spaces and its properties associated to the weight function \(A\) (such as weighted Sobolev inequalities and embeddings into classical and weighted Lebesgue spaces) are heavily used throughout this work and such tools are available to the study of equations like \eqref{PDE-ex} (for instance the Caffarelli-Kohn-Nirenberg inequality \cite{CKN1984} for \(\abs{x}^\alpha\) or what was done in \cite{CR-O2013-1,Cas2016-2} for monomials \(x^A\)). This later extension to higher dimension is something we are interested in and this article could be thought as a stepping stone towards that goal.

As it was established in \cite{CW11-1} it is convenient to separate the main results regarding \eqref{SL-eq} into two cases depending on the behavior near the origin we want to prescribe: the Dirichlet problem and the Neumann problem.

\subsection{Dirichlet problem}

We first consider \eqref{SL-eq} under the condition \(\lim\limits_{x\to 0^+}u(x)=0\). As it was observed in \cite{CW11-1} this is only possible for \(\alpha<\frac12\). We recover the results obtained for the case \(f\in L^2\) and \(0<\alpha<\frac12\) in \cite{CW11-1} and extend them for \(f\in L^p\) and every \(\alpha<\frac12\)

\begin{theorem}\label{Thm-DIR}
	For any given \(\alpha<\frac12\) and \(f\in L^p(0,1)\) with \(1\leq p\leq \infty\) there exists a unique function \(u_{D}\in W^{2,p}_{loc}((0,1])\) satisfying \eqref{SL-eq} a.e. and the following properties:
	\begin{enumerate}[label=(\roman*)]
		\item \(u_{D}\in L^p(0,1)\) with \(\norm{u_{D}}_{L^p}\leq C\norm{f}_{L^p}\),
		\item  $x^{2\alpha}u_{D}', x^{2\alpha-1}u_{D}\in W^{1,p}(0,1)$ with 
		\[
		\|x^{2\alpha}u_{D}'\|_{W^{1,p}}+\|x^{2\alpha-1}u_{D}\|_{W^{1,p}}\leq C \|f\|_{L^p},
		\]
		\item $x^{2\alpha}u_{D}\in W^{2,p}(0,1)$ with $\|x^{2\alpha}u_{D}\|_{W^{2,p}}\leq C \|f\|_{L^p}.$
	\end{enumerate}
	
	Additionally, \(u_D\) satisfies 
	\begin{enumerate}[resume, label=(\roman*)]
		\item $u_{D}\in C^{0,\frac{1}{2}-\alpha}[0,1],$ with $\|u_{D}\|_{C^{0,\frac{1}{2}-\alpha}}\leq C \|f\|_{L^p}$ when \(0<\alpha<\frac12\).
		\item $u_{D}\in C^{0,\frac{1}{2}}[0,1],$ with $\|u_{D}\|_{C^{0,\frac{1}{2}}}\leq C \|f\|_{L^p}$ when \(\alpha<0\).
		\item \(\lim\limits_{x\to 0^+}u_D(x)=0\), and in fact one has $\abs{x^{2\alpha-1}u_{D}(x)}\leq Cx^{\frac1p}$.
	\end{enumerate}
\end{theorem}

\begin{remark}
	In the above theorem and throughout the rest of this work, the constant \(C>0\) will denote a universal constant depending on the parameters of the problem \(\alpha\) and \(p\), but not on \(f\). Also, the value of \(C\) might change from one line to the next. 
\end{remark}

\begin{remark}\label{obs1}
	Observe that from \cref{Thm-DIR} we have that $(x^{2\alpha}u_{D}')'=2\alpha x^{2\alpha-1}u_{D}' +x^{2\alpha}u_{D}''\in L^p(0,1)$ but it is not necessarily true that both $x^{2\alpha-1}u_{D}'$ and $x^{2\alpha}u_{D}''$ belong to $L^p(0,1)$. This can be seen from the fact that there exists a function $f\in C_c^{\infty}(0,1)$ such that, near the origin, the solution given by \cref{Thm-DIR} expands as
	$$
	u_{D}(x)=a_1x^{1-2\alpha}+a_2x^{3-4\alpha}+a_3x^{5-6\alpha}+\cdots.
	$$
	Therefore $x^{2\alpha-1}u_{D}' \sim x^{2\alpha}u_{D}'' \sim x^{-1} \notin L^p(0,1)$ for any \(1\leq p\leq \infty\).
\end{remark}

\begin{remark} \label{obsnueva}
	The boundary behavior near the origin is optimal in the following sense. If we define
	$$
	K_D(x):=\sup_{\|f\|_{L^p}\leq1}\left|x^{2\alpha-1}u_{D}(x)\right|,
	$$
	there exist $\ve_0>0,\ C_1 >0$, and $C_2>C_1$ such that, for all $0<x\leq \ve_0$ one has
	\[
	C_1\leq K_D(x)\leq C_{2}.
	\]
	This fact will be shown below in \cref{pf-rem1}.
\end{remark}

\subsection{Neumann problem}

If we study \eqref{SL-eq} under the (weighted) Neumann boundary condition \(\lim\limits_{x\to 0^+}x^{2\alpha}u'(x)=0\) we have the following generalization of the results from \cite{CW11-1}.

\begin{theorem}\label{Thm-NEU}
	For \(\alpha<1\) and \(f\in L^p(0,1)\) with \(1\leq p\leq \infty\) there exists a unique function \(u_{N}\in W^{2,p}_{loc}((0,1])\) satisfying \eqref{SL-eq} and the following
	\begin{enumerate}[label=(\roman*)]
		\item \(u_N\in L^p\) with \(\norm{u_N}_{L^p}\leq C \norm{f}_{L^p}\),
		\item \(\lim\limits_{x\to 0^+}x^{2\alpha}u_{N}'(x)=0\), and in fact one has $\lim_{x\to 0^+}x^{2\alpha-1+\frac{1}{p}}u_{N}'(x)=0$,
		\item \(x^{2\alpha}u_{N}'\in W^{1,p}(0,1)\) with \(\norm{x^{2\alpha}u_{N}'}_{W^{1,p}}\leq C\norm{f}_{L^p}\),
		\item\label{rem-p=1} If \(p>1\) then $x^{2\alpha-1}u_{N}'\in L^p(0,1)$ and $x^{2\alpha}u_{N}''\in L^p(0,1)$, with
		\[
		\|x^{2\alpha-1}u_{N}'\|_{L^p}+\|x^{2\alpha}u_{N}''\|_{L^p}\leq C\|f\|_{L^p}.
		\]
	\end{enumerate}
	Additionally we have
	\begin{enumerate}[resume,label=(\roman*)]
		\item $u_{N}(x)\in C^{0,\frac{1}{2}-\alpha}[0,1],$ with $\|u_{N}\|_{C^{0,\frac{1}{2}-\alpha}}\leq C \|f\|_{L^p}$, if \(0<\alpha<1\)
		\item $u_{N}\in W^{1,p}(0,1)$ with $\|u_{N}\|_{W^{1,p}}\leq C\|f\|_{L^p}$ if \(\alpha\leq 0\). 
	\end{enumerate}
\end{theorem}

\begin{remark}\label{rem-p=12}
	Unlike the Dirichlet problem, each of the terms in the expansion of $(x^{2\alpha}u'(x))'=2\alpha x^{2\alpha-1}u'(x)+x^{2\alpha}u''(x)$ belong to $L^p(0,1)$ when \(u=u_{N}\) and \(p>1\). However, this is no longer true for \(p=1\). We will see this in \cref{rem-p=13} by giving an explicit example of a function \(f\in L^1\) for which \(x^{2\alpha-1}u_N'\notin L^1\).
\end{remark}

\begin{remark} \label{obsN2}
	The boundary behavior $\lim\limits_{x\to 0^+}x^{2\alpha-1+\frac{1}{p}}u_{N}'(x)=0$ is optimal in the following sense: there exists a $\ve_0$ such that and constants \(0<C_1<C_2\) such that
	$$
	C_1\leq K_N(x):=\sup_{\|f\|_{L^p}\leq1}\left|x^{2\alpha-1+\frac{1}{p}}u_{N}'(x)\right|\leq C_2
	$$
	for all $0<x\leq \ve_0$. This fact will be shown below in \cref{pf-rem2}.
\end{remark}

The rest of this article consists on the proof of these results. To make the exposition clear we have divided each proof into the following parts: in \cref{sect-unique} we prove the uniqueness of both types of solutions and then in \cref{sect-existence} we prove the existence part for both \cref{Thm-DIR,Thm-NEU}. Then in \cref{reg-Dir} we prove the regularity properties of the solution \(u_D\) mentioned in \cref{Thm-DIR} to then study the solution \(u_N\) given in \cref{Thm-NEU} in \cref{reg-Neu}. We conclude this paper with some results about Bessel's function in the \cref{app-bessel}.

\section{Uniqueness in both \cref{Thm-DIR,Thm-NEU}}\label{sect-unique}

This was done for \(\alpha>0\) in \cite{CW11-1} and the same argument remains valid for \(\alpha<0\). We include the main steps on each proof for the reader's convenience.

\subsection{The Dirichlet problem}

\begin{proposition}\label{uniqD}
	If $\alpha<\frac12$ and $u\in W^{2,p}_{loc}(0,1)$ for some \(1\leq p\leq\infty\) satisfies
	\begin{equation}\label{eq1D}
		\left\{
		\begin{aligned}
			-(x^{2\alpha}u'(x))'+u(x)&=0 && \text{in } (0,1],\\
			u(1)&= 0,\\
			\lim_{x\to 0^+} u(x)&= 0,
		\end{aligned}
		\right.
	\end{equation}
	then $u\equiv0$.
\end{proposition}

\begin{proof}	
	Observe that since \(W^{2,p}_{loc}(0,1)\hookrightarrow C^0(0,1)\) and because $\lim\limits_{x\to 0^+}u(x)=u(1)=0$ we have $u\in C^0([0,1])$. Additionally, for every \(0<x<y<1\) we can write
	$$
	y^{2\alpha}u'(y)-x^{2\alpha}u'(x)=\int_x^y(s^{2\alpha}u'(s))'ds=\int_x^yu(s)ds,
	$$
	so that the function \(x^{2\alpha}u'(x)\) is continuous in \([0,1]\), in particular if we multiply \eqref{eq1D} by $u$ and integrate we have that for all $0<\ve<1$
	\begin{align*}
		0&=-\int_\ve^{1-\ve}(x^{2\alpha}u'(x))'u(x)dx+\int_\ve^{1-\ve}u^2(x)dx\\
		&=\int_\ve^{1-\ve}\abs{x^{\alpha}u'(x)}^2dx+\int_\ve^{1-\ve} u^2(x)dx-x^{2\alpha}u'(x)u(x)\Big|_{x=\ve}^{x=1-\ve}\\
		&\xrightarrow[\ve\to 0]{}\int_0^1\abs{x^{\alpha}u'(x)}^2dx+\int_0^1 u^2(x)dx,
	\end{align*}
	therefore $u=0$ a.e., and because \(u\in C^0([0,1])\) in fact we deduce that \(u\equiv 0\).
	
\end{proof}

\subsection{The Neumann problem}

\begin{proposition}\label{uniqN}
	If $\alpha<1$ and $u\in W^{2,p}_{loc}(0,1)$ for some \(1\leq p\leq\infty\) satisfies
	\begin{equation}
		\left\{
		\begin{aligned}
			-(x^{2\alpha}u'(x))'+u(x)&=0 && \text{in } (0,1],\\
			u(1)&= 0,\\
			\lim_{x\to 0^+} x^{2\alpha}u'(x)&= 0,
		\end{aligned}
		\right.
	\end{equation}
	then $u\equiv0$.
\end{proposition}

\begin{proof}
	Suppose for a moment that \(u\in C^0([0,1])\), if that is the case then just as in the proof of \cref{uniqD} we can write
	\begin{align*}
		0&=-\int_\ve^{1-\ve}(x^{2\alpha}u'(x))'u(x)dx+\int_\ve^{1-\ve}u^2(x)dx\\
		&=\int_\ve^{1-\ve}\abs{x^{\alpha}u'(x)}^2dx+\int_\ve^{1-\ve} u^2(x)dx-x^{2\alpha}u'(x)u(x)\Big|_{x=\ve}^{x=1-\ve}\\
		&\xrightarrow[\ve\to 0]{}\int_0^1\abs{x^{\alpha}u'(x)}^2dx+\int_0^1 u^2(x)dx,
	\end{align*}
	and therefore \(u\equiv 0\).
	
	So we only need to show that \(u\in C^0([0,1])\). Observe that from the equation and the fact that \(u\in W^{2,p}_{loc}(0,1)\) we readily obtain that \(u\in C^1((0,1])\). Additionally since \(\lim\limits_{x\to 0^+}x^{2\alpha}u'(x)=0\) we deduce that \(x^{2\alpha}u'(x)\) is bounded in \([0,1]\), hence
	\[
	\abs{u(y)-u(x)}=\abs{\int_x^yu'(s)ds}\leq C\int_x^ys^{-2\alpha}ds\leq C\abs{y^{1-2\alpha}-x^{1-2\alpha}},
	\] 
	thus \(u\in C^0([0,1])\) for any \(\alpha<\frac12\). To prove that \(u\in C^0([0,1])\) for \(\frac12\leq \alpha<1\) the argument is similar to what we will do in \cref{ssect-iter}. We refer the reader to the proof of \cite[Theorems 1.8 and 1.12]{CW11-1} for the detailed argument for the uniqueness.
\end{proof}

\section{Existence of solutions}\label{sect-existence}

The solutions given by \cref{Thm-DIR,Thm-NEU} can be characterized as follows. For each \(\alpha<1\) we define the space
$$
X^{\alpha}=\left\{\,u\in H^1_{loc}(0,1)\,:\, u \in L^2(0,1)\,\text{and}\, x^{\alpha}u'\in L^2(0,1)\right\},
$$
and equip it with the inner product
$$
(u,v)_{X^\alpha}=\int_0^1\left(x^{2\alpha}u'(x)v'(x)+u(x)v(x)\right)dx,
$$
which makes $X^{\alpha}$ a Hilbert space. Observe that for \(\alpha=0\) then \(X^\alpha=H^1(0,1)\) the classical Sobolev space, where it is clear that if \(\alpha<0\) then \(X^\alpha\hookrightarrow H^1(0,1)\) and if \(\alpha>0\) then \(H^1(0,1)\hookrightarrow X^\alpha\), both inclusions being bounded. Moreover, it is also clear that for any \(\alpha\in \RR\) we have the embedding \(X^\alpha\hookrightarrow H^1_{loc}((0,1])\) and in particular, elements in \(X^\alpha\) are continuous near \(x=1\), therefore it makes sense to define the closed subspace
\[
X^\alpha_{\cdot 0}=\set{u\in X^\alpha(0,1): u(1)=0}
\]
which is a natural space to find weak solutions to \eqref{SL-eq} with (weighted) Neumann type behavior near the origin. Additionally, it is known from \cite[Appendix]{CW11-1} that if \(0<\alpha<\frac12\) then \(X^\alpha\hookrightarrow C^{\frac12-\alpha}([0,1])\) and because  \(X^\alpha\hookrightarrow H^1(0,1)\hookrightarrow C^{\frac12}([0,1])\) for \(\alpha<0\) we can also define the closed subspace
\[
X^\alpha_{00}=\set{u\in X^\alpha: u(1)=u(0)=0}
\]
for any \(\alpha<\frac12\), which gives another natural space to find solutions to \eqref{SL-eq}, this time with Dirichlet boundary behavior near the origin.

With the above in mind we consider weak solutions to \eqref{SL-eq} as solutions to
\begin{equation}\label{w-SL-eq}
	\int_0^1\pt{x^{2\alpha}u'(x)v'(x)+u(x)v(x)}dx=\int_0^1f(x)v(x)dx,\quad\forall\, v\in X,
\end{equation}
where \(X\) is either \(X^\alpha_{\cdot 0}\) or \(X^\alpha_{00}\).

In order to have existence of solutions to \eqref{w-SL-eq}, Riesz's theorem tells us that it would suffice that \(f\) belongs to the dual space of \(X\), in particular, given \(f\in L^p(0,1)\) for \(1\leq p\leq \infty\) we should analyze whether the functional
$$
\begin{array}{rrcl}
	\varphi_f : & X^{\alpha}(0,1)& \longrightarrow & \RR \\
	& v & \longmapsto & \varphi_f(v)=\displaystyle\int_0^1 f(x)v(x)dx
\end{array},
$$
is bounder or not. On the one hand, because \(X^\alpha(0,1)\hookrightarrow L^\infty(0,1)\) for every \(\alpha<\frac12\) it follows that \(\varphi_f\in \pt{X^\alpha}^*\) for every \(1\leq p\leq \infty\) with
\[
\norm{\varphi_f}_{\pt{X^\alpha}^*}\leq C\norm{f}_{L^p},
\]
and on the other hand for \(\alpha\geq \frac12\) we know from \cite{CW11-1} that
\begin{itemize}
	\item If \(\alpha=\frac12\) then \(X^{\alpha}\hookrightarrow L^q(0,1)\) for \(1\leq q<\infty\),
	\item If \(\frac12<\alpha\leq 1\) then \(X^\alpha\hookrightarrow L^q(0,1)\) for \(1\leq q\leq \frac{2}{2\alpha-1}\)
\end{itemize}
from where we deduce that
\begin{itemize}
	\item If \(\alpha=\frac12\) then \(\varphi_f\in \pt{X^{\alpha}}^*\) for \(1<p\leq \infty\),
	\item If \(\frac12<\alpha\leq 1\) then \(\varphi_f\in \pt{X^{\alpha}}^*\) for \(\frac{2}{3-2\alpha}\leq p\leq \infty\).
\end{itemize}

Therefore, either for \(\alpha<\frac12\) and every \(1\leq p\leq \infty\) or for \(\frac12\leq \alpha<1\) and \(p\) as above, we readily have the existence of weak solutions to \eqref{w-SL-eq} in either \(X^\alpha_{00}\) or \(X^\alpha_{\cdot 0}\) thus proving the existence part of \cref{Thm-DIR} and part of the existence for \cref{Thm-NEU}. However to handle the existence part for \cref{Thm-NEU} in each of the remaining cases we will have to look for solutions elsewhere.

For \(1\leq q\leq \infty\) we consider the space
$$
X^{\alpha,q}=\left\{\,u\in W^{1,q}_{loc}(0,1)\,:\, u \in L^q(0,1)\,\text{and}\, x^{\alpha}u'\in L^q(0,1)\right\},
$$
which is a reflexive Banach space when \(1<q<\infty\) when equipped with the norm
\[
\norm{u}_{X^{\alpha,q}}^q=\norm{x^\alpha u'}_{L^q}^q+\norm{u}_{L^q}^q.
\]
Additionally we consider its closed subspace
\[
X^{\alpha,q}_{\cdot 0}=\set{u\in X^{\alpha,q}(0,1): u(1)=0},
\]
and equip it with the equivalent norm (thanks to \cite[Theorem A.2]{CW11-1})
\[
\norm{u}_{X^{\alpha,q}_{\cdot 0}}=\norm{x^\alpha u'}_{L^q}.
\]

\subsection{Existence part of \cref{Thm-NEU}: the case \(\frac12\leq\alpha<1\) and \(1\leq p<\frac{2}{3-2\alpha}\)}

To deal with this case we will argue by duality with the help of the following

\begin{proposition}\label{dual-prop}
	Suppose that \(\frac12\leq\alpha<1\) and let \(g\in L^s(0,1)\) with \(s>\frac{1}{1-\alpha}\). Then there exists \(w\in X^\alpha_{\cdot 0}\) solution to
	\begin{equation}\label{dual-eq}
		-(x^{2\alpha}w'(x))'+w(x)=-(x^\alpha g(x))'\qquad \text{in }(X^\alpha_{\cdot 0})^*,
	\end{equation}
	satisfying in addition \(\norm{w}_{L^\infty}\leq C\norm{g}_{L^s}\).
\end{proposition}

Before proving this proposition let us use it to analyze the existence part of \cref{Thm-NEU} for case \(1\leq p<\frac{2}{3-2\alpha}\) and \(\frac12\leq \alpha<1\). Take $f \in L^p(0,1)$ and since \(\frac{2}{3-2\alpha}<2\) we can construct a sequence $f_n\in L^2(0,1) $ such that \(f_n\xrightarrow[]{n\to\infty}f\) strongly in \(L^p\). From the \(L^2\) theory at the beginning of this section we know that there exists $u_n \in X^{\alpha}_{\cdot 0}$ weak solution to \eqref{SL-eq} with \(f_n\) as the right hand side. We claim that \(u_n\) converges to some \(u\) solution of \eqref{SL-eq} with \(f\) as its right hand side. Indeed, fix \(s>\frac1{1-\alpha}\) and consider arbitrary \(g\in L^s\) with \(w=w_g\) the solution given by \cref{dual-prop}. We then use \(\psi=u_n-u_m\) as a test function in \eqref{dual-eq}, that is in
 \[
\int_0^1 x^{2\alpha}w'(x)\psi'(x) dx + \int_0^1 w(x)\psi(x)dx = \int_0^1 x^{\alpha}g(x)\psi'(x)dx
\]
and we observe that thanks to \cref{dual-prop} we have

\begin{align*}
	\left| \int_0^1 x^{\alpha}\psi'(x)g(x)dx\right| &=\left|\int_0^1 x^{2\alpha}w'(x)\psi'(x) dx + \int_0^1 w(x)\psi(x)dx\right| \\
	&= \left|\int_0^1 (-(x^{2\alpha}\psi'(x))'+\psi(x)) w(x) dx\right|  \\
	& \leq  \left\| w\right\|_{L^\infty} \left\| -(x^{2\alpha}\psi')'+\psi\right\|_{L^p}  \\
	& \leq  C \left\| g\right\|_{L^s} \left\| f_n - f_m\right\|_{L^p}.
\end{align*}

Therefore
\begin{equation} \label{L1}
	\left\| x^{\alpha}\psi'\right\|_{L^{s'}}
	= \sup_{ \left\| g\right\|_{L^{s}}=1} \left| \int_0^1 x^{\alpha}\psi'(x)g(x)dx\right|
	\leq C\left\| f_n - f_m\right\|_{L^p},
\end{equation} 
and in particular the sequence \((u_n)\) is a Cauchy sequence in \(X^{\alpha,s'}_{\cdot 0}\) and therefore we can find $u\in X^{\alpha,s'}_{\cdot 0}$ such that $(u_n)$ converges to $u$ in $X^{\alpha,s'}_{\cdot 0}$. Furthermore we can pass to the limit in
\[
\int_0^1 x^{2\alpha}u_n'(x)v'(x)dx+\int_0^1 u_n(x)v(x)dx=\int_0^1 f_n(x)v(x)dx \qquad \forall\, v\in C^1_c([0,1))
\]
to deduce that $-(x^{2\alpha}u')' + u = f$ a.e. in \((0,1)\) and that
\[
\left\| u\right\|_{X^{\alpha,s'}_{\cdot 0}} \leq C_{\alpha,1} \left\| f\right\|_{L^p}.
\]
Since \(s>\frac1{1-\alpha}\) was arbitrary we conclude that in fact \(u\in X^{\alpha,q}_{\cdot 0}\) for any \(1\leq q<\frac1\alpha\) with
\[
\left\| u\right\|_{X^{\alpha,q}_{\cdot 0}} \leq C\left\| f\right\|_{L^p},
\]
and in particular \(\norm{u}_{L^q}\leq C\norm{f}_{L^p}\).

\begin{proof}[Proof of \cref{dual-prop}]
	Observe that for \(\alpha\geq \frac12\) it holds that \(\frac{1}{1-\alpha}\geq 2\), in particular if \(g\in L^s(0,1)\) for \(s>\frac1{1-\alpha}\) then \(g\in L^2(0,1)\). Therefore the functional \(\varphi\mapsto \int_0^1 x^{\alpha}g(x)\varphi'(x)dx\) is a well defined element of the dual space of \(X^\alpha_{\cdot 0}\) and Riesz's theorem guarantees the existence and uniqueness of a function \(w\in X^\alpha_{\cdot 0}\) such that \begin{equation}\label{iter1}
		\int_0^1 x^{2\alpha}w'(x)\varphi'(x) dx + \int_0^1 w(x)\varphi(x)dx = \int_0^1 x^{\alpha}g(x)\varphi'(x)dx,\qquad\forall\, \varphi\in X_{\cdot 0}^{\alpha}
	\end{equation}
	and
	\[
	\norm{w}_{L^2}\leq \norm{w}_{X^\alpha_{\cdot 0}}\leq \norm{g}_{L^2}\leq \norm{g}_{L^s}.
	\]
	
	So what remains to be shown is that in fact \(w\in L^\infty(0,1)\) with 
	\[
	\norm{w}_{L^\infty}\leq C\norm{g}_{L^s}.
	\]
	For this purpose we write \(s=\frac{2}{(2-\ve)(1-\alpha)}\) for some \(0<\ve<2\) and to prove the boundedness of the solution \(w\) we will use Moser's iteration method as in \cite{Serrin1964}, that is for $m\geq 1$ and $0\leq k \leq l$ we define $F:[k, \infty) \rightarrow \mathbb{R}$ and we consider 
	\begin{equation}
		F(x)=F_{m,k,l}(x)=\begin{dcases}
			x^m &\text{if } k\leq x\leq l,\\
			l^{m-1}(mx-(m-1)l)& \text{if } x > l.
		\end{dcases} 
	\end{equation}
	Observe that $F\in C^1[k,\infty)$ with $|F'(x)|\leq m l^{m-1}$ and that if we define $G:\mathbb{R}\rightarrow \mathbb{R}$ as
	\begin{equation}
		G(x)= \text{sign}(x)\left(F(\overline{x})|F'(\overline{x})|-mk^{\beta}\right)
	\end{equation}
	where $\beta=2m-1$ and $\overline{x}=|x|+k$. Observe that $F$ and $G$ satisfy
	\begin{gather*}
		|G|\leq F(\overline{x})|F'(\overline{x})|,\\
		\overline{x}F'(\overline{x})\leq m F(\overline{x}),\\
		G'(x)=\begin{dcases}
			\frac{\beta}{m} |F'(\overline{x})|^2&\text{if } |x|<l-k,\\
			|F'(\overline{x})|^2& \text{if } |x|>l-k.
		\end{dcases}
	\end{gather*}
	
	To continue, we take $\eta \in C_c^{\infty}([0,1))$ and observe that $\varphi = \eta^2 G(w)$ belongs to \(X^\alpha_{\cdot 0}\) so it is a valid test function in \eqref{iter1}. Since \(0\leq x^\alpha\leq 1\) in \((0,1)\) we have
	\begin{align*}
		\pt{x^{2\alpha}w'-x^{\alpha}g}\varphi'+w\varphi
		&=\eta^2 G'(w)\pt{x^{2\alpha}w'-x^{\alpha}g}
		w'+2\eta G(w)\eta'\pt{x^{2\alpha}w'-x^{\alpha}g}\\
		&\qquad +\eta^2 G(w)w \\
		&\geq \eta^2 G'(w)\left(\frac{1}{2}\abs{x^\alpha w'}^2-\frac{1}{k^2}\abs{g}^2\overline{w}^2\right)\\
		&\qquad-2\eta\abs{\eta' G(w)}\pt{x^{2\alpha}|w'|+\frac{1}{2k}\abs{g}\overline{w}}\\
		&\qquad -\eta^2 |G(w)|\overline{w}\\
		&=\frac{1}{2}\eta^2 G'(w)\abs{x^{\alpha}w'}^2-2\eta |G(w)\eta'|x^{2\alpha}|w'|\\
		&\qquad -\frac{2}k\eta |G(w)\eta'|\abs{g}\overline{w}\\
		&\qquad -\pt{\frac{1}{2k^2}\abs{g}^2\eta^2 G'(w)\overline{w}^2+\eta^2 |G(w)|\overline{w}}\\
		&\geq\frac{1}{2}\abs{x^{\alpha}\eta F'(\overline{w})w'}^2-2|\eta' F(\overline{w})|\,|x^\alpha\eta F'(\overline{w})w'|\\
		&\qquad -\frac2k\abs{g}|\eta' F(\overline{w})|\,|\eta F'(\overline{w})\overline{w}|\\
		&\qquad -\pt{\frac{\beta}{2m k^2}\abs{g}^2|\eta \overline{w}F'(\overline{w})|^2+|\eta F(\overline{w})|\,|\eta F'(\overline{w})\overline{w}|}\\
		&\geq\frac{1}{2}\abs{x^{\alpha}\eta F'(\overline{w})w'}^2-2|\eta' F(\overline{w})|\,|x^\alpha\eta F'(\overline{w})w'|\\
		&\qquad -\frac{2m}k\abs{g}|\eta' F(\overline{w})|\,|\eta F(\overline{w})|\\
		&\qquad -m\pt{\frac{\beta}{2k^2}\abs{g}^2+1}|\eta F(\overline{w})|^2
	\end{align*}
	By denoting $v=F(\overline{w})$ the above can be written as
	\begin{multline*}
	\pt{x^{2\alpha}w'-x^{\alpha}g}\varphi'+w\varphi
	\geq
	\frac{1}{2}\abs{x^{\alpha}\eta v'}^2
	-2|\eta' v|\,|x^\alpha\eta v'|
	-\frac{2m}k\abs{g}|\eta' v|\,|\eta v|\\
	-m\beta\pt{\frac{1}{2k^2}\abs{g}^2+1}|\eta v|^2
	\end{multline*}
	because \(\beta=2m-1\geq 1\) since \(m\geq 1\). If we integrate the above and we use \eqref{iter1} we obtain
	\begin{multline}\label{base-ineq}
		\int_0^1|\eta v'|^2 x^{2\alpha} dx \leq 4\int_0^1|\eta' v|\,|x^\alpha\eta v'|dx
		+\frac{4m}k\int_0^1\abs{g}|\eta' v|\,|\eta v|dx\\
		+2m\beta\int_0^1\pt{\frac{1}{2k^2}\abs{g}^2+1}|\eta v|^2dx
	\end{multline}
	
	To continue we divide the analysis into two cases \(\frac12<\alpha<1\) and \(\alpha=\frac12\). Firstly we consider \(\frac12<\alpha<1\) and we note that by our hypotheses the following estimates holds
	$$
	\int_0^1 |\eta' v|\,|x^{\alpha}\eta v'|dx\leq \|\eta' v\|_{L^2}\|x^{\alpha}\eta v'\|_{L^2},
	$$
	and if we select \(k=\norm{g}_{L^s}\) and write $2^*=\frac{2}{2\alpha-1}$ then by using Hölder's inequality for \(1=\frac{1}{r}+\frac12+\frac1{2^*}\) and \cite[Theorem A.2]{CW11-1} we obtain
	\begin{align*}
		\frac1k\int_0^1 \abs{g}|\eta' v|\,|\eta v|dx&\leq\frac1k\|g\|_{L^{r}} \|\eta' v\|_{L^2}\|\eta v\|_{L^{2^*}}\\
		&\leq\frac1k\|g\|_{L^{\frac1{1-\alpha}}} \|\eta' v\|_{L^2}\|\eta v\|_{L^{2^*}}\\
		&\leq \frac{C}k\|g\|_{L^{s}}\|\eta' v\|_{L^2}\left(\|x^{\alpha}\eta' v\|_{L^2}+\|x^{\alpha}\eta v'\|_{L^2}\right)\\
		& \leq C\pt{\|\eta' v\|_{L^2}^2+\|\eta' v\|_{L^2}\|x^{\alpha}\eta v'\|_{L^2}}.
	\end{align*}
	If we write \(\tilde g=1+\frac{\abs{g}^2}{2k^2}\) then
	\begin{align*}
		\int_0^1 \tilde g|\eta v|^2 dx&
		\leq\|\tilde g\|_{L^{\frac{s}2}}\|\eta v\|_{L^2}^{\ve}\,\|\eta v\|_{L^{2^*}}^{2-\ve}\\
		& \leq C\left(1+\norm{\frac{g}{k}}_{L^{s}}^s\right)^{\frac2s} \|\eta v\|_{L^2}^{\ve}\left(\|\eta' v\|_{L^2}^{2-\ve}+\|x^{\alpha}\eta v'\|_{L^2}^{2-\ve}\right)\\
		&\leq C \|\eta v\|_{L^2}^{\ve}\left(\|\eta' v\|_{L^2}^{2-\ve}+\|x^{\alpha}\eta v'\|_{L^2}^{2-\ve}\right).
	\end{align*}
	Therefore we reach the following
	\begin{multline}\label{end-ineq}
	\|x^{\alpha}\eta v'\|_{L^2}^2\leq C \left(\|\eta' v\|_{L^2}\|x^{\alpha}\eta v'\|_{L^2}+m(\|\eta' v\|_{L^2}^2+\|\eta' v\|_{L^2}\|x^{\alpha}\eta v'\|_{L^2})\right.\\
	\left.+2m\beta \|\eta v\|_{L^2}^{\ve}\left(\|\eta' v\|_{L^2}^{2-\ve}+\|x^{\alpha}\eta v'\|_{L^2}^{2-\ve}\right) \right).
	\end{multline}
	
	Now we take $z=\frac{\|x^{\alpha}\eta v'\|_{L^2}}{\|\eta' v\|_{L^2}}$ and $\zeta=\frac{\|\eta v\|_{L^2}}{\|\eta v'\|_{L^2}}$ so the above can be written as 
	$$z^2\leq C(z+m(1+z)+2m(2m-1)\zeta^{\ve}(1+z^{2-\ve})),$$
	and thanks to \cite[Lemma 2]{Serrin1964} we obtain 
	$$z\leq C m^{\frac{2}{\ve}}(1+\zeta),$$
	which translates to
	$$\|x^{\alpha}\eta v'\|_{L^2} \leq  C m^{\frac{2}{\ve}}\left(\|\eta' v\|_{L^2} + \|\eta v\|_{L^2}\right),$$
	and because \(\abs{(\eta v)'}\leq \abs{\eta v'}+\abs{\eta' v}\) the above also gives
	$$\|x^{\alpha}(\eta v)'\|_{L^2} \leq  C m^{\frac{2}{\ve}}\left(\|\eta' v\|_{L^2} + \|\eta v\|_{L^2}\right).$$
	Hence we can use \cite[Theorem A.2]{CW11-1} and deduce that
	\begin{equation}\label{ineq1}
		\left(\int_0^1|\eta v|^{2\chi} dx\right)^{\frac{1}{2\chi}}\leq C m^{\frac{2}{\ve}}\left(\left(\int_0^1|\eta v|^2 dx\right)^{\frac{1}{2}}+ \left(\int_0^1|\eta' v|^2 dx\right)^{\frac{1}{2}} \right),
	\end{equation}
	where $\chi=\frac{1}{2\alpha-1}$.
	
	If \(\alpha=\frac12\) then \(s=\frac{4}{2-\ve}\). We start from \eqref{base-ineq} and estimate the terms on the right hand side as follows: we still have
	$$
	\int_0^1 |\eta' v|\,|x^{\alpha}\eta v'|dx\leq \|\eta' v\|_{L^2}\|x^{\alpha}\eta v'\|_{L^2},
	$$
	but for the second term we recall that \(\norm{u}_{L^t}\leq C\norm{x^\alpha u'}_{L^2}\) for any \(t<\infty\) when \(\alpha=\frac12\), hence we can write
	\begin{align*}
		\frac1k\int_0^1 \abs{g}|\eta' v|\,|\eta v|dx&\leq\frac1k\|g\|_{L^{s}} \|\eta' v\|_{L^{\frac{8}{4+\ve}}}\|\eta v\|_{L^{\frac{8}{\ve}}}\\
		&\leq \frac{C}k\|g\|_{L^{s}}\|\eta' v\|_{L^2}\left(\|x^{\alpha}\eta' v\|_{L^2}+\|x^{\alpha}\eta v'\|_{L^2}\right)\\
		& \leq C\|\eta' v\|_{L^2}^2+\|\eta' v\|_{L^2}\|x^{\alpha}\eta v'\|_{L^2},
	\end{align*}
	and for the third term we obtain
	\begin{align*}
		\int_0^1 \tilde g|\eta v|^2 dx&
		\leq\|\tilde g\|_{L^{\frac{s}2}}\|\eta v\|_{L^{\frac{8}{4+\ve}}}^{\frac{\ve}2}\,\|\eta v\|_{L^{\frac{8}{\ve}}}^{2-\frac{\ve}2}\\
		& \leq C\left(1+\norm{\frac{g}{k}}_{L^{s}}^s\right)^{\frac2s} \|\eta v\|_{L^2}^{\frac{\ve}2}\left(\|\eta' v\|_{L^2}^{2-\frac{\ve}2}+\|x^{\alpha}\eta v'\|_{L^2}^{2-\frac{\ve}2}\right)\\
		&\leq C \|\eta v\|_{L^2}^{\frac{\ve}2}\left(\|\eta' v\|_{L^2}^{2-\frac{\ve}2}+\|x^{\alpha}\eta v'\|_{L^2}^{2-\frac{\ve}2}\right).
	\end{align*}
	so by repeating what we did for the case \(\frac12<\alpha<1\) we reach \eqref{end-ineq} but with \(\frac\ve2\) instead of \(\ve\), so instead of \eqref{ineq1} we reach
	\begin{equation}\label{ineq2}
		\left(\int_0^1|\eta v|^{2\chi} dx\right)^{\frac{1}{2\chi}}\leq C m^{\frac{4}{\ve}}\left(\left(\int_0^1|\eta v|^2 dx\right)^{\frac{1}{2}}+ \left(\int_0^1|\eta' v|^2 dx\right)^{\frac{1}{2}} \right),
	\end{equation}
	for \(\chi=\frac4\ve>1\).
	
	In summary, for any \(\frac12\leq \alpha<1\) we have that
	\begin{equation}\label{ineq3}
		\left(\int_0^1|\eta v|^{2\chi} dx\right)^{\frac{1}{2\chi}}\leq C m^{\frac{L}{\ve}}\left(\left(\int_0^1|\eta v|^2 dx\right)^{\frac{1}{2}}+ \left(\int_0^1|\eta' v|^2 dx\right)^{\frac{1}{2}} \right).
	\end{equation}
	where \(L=2\) or \(L=4\) and the appropriate \(\chi>1\).
	
	We now proceed to select the cut-off function \(\eta\). For each \(n\in\NN\cup\set{0}\) we write $I_n=[0,\frac{1}{2}+2^{-n-1})$ and take $\eta_n \in C_c^{\infty}(I_n)$ in such a way that $\eta_n \equiv 1 $ in $I_{n+1}$, and that $|\eta_n'|\leq C 2^n$. If we use such \(\eta_n\) in \eqref{ineq3} and we pass to the limit $l\to \infty$ we deduce that
	$$
	\left(\int_{I_{n+1}}|\overline{w}|^{2m\chi} dx\right)^{\frac{1}{2m\chi}}\leq C^{\frac1m} 2^{\frac{n}m} m^{\frac{L}{m\ve}}\left(\int_{I_n}|\overline{w}|^{2m} dx\right)^{\frac{1}{2m}},
	$$
	for each $m\geq 1$. In particular, if we take $m_n = \chi^n$ and \(s_n=2\chi^{n}\) we obtain 
	$$
	\left(\int_{I_{n+1}}|\overline{w}|^{s_{n+1}} dx\right)^{\frac{1}{s_{n+1}}}\leq C^{\chi^{-n}} 2^{n \chi^{-n}} \chi^{\frac{L}{\ve} \chi^{-n}}\left(\int_{I_n}|\overline{w}|^{s_{n}} dx\right)^{\frac{1}{s_n}}.
	$$
	Since $\chi>1$ we know that both $\sum_{k=0}^{\infty} k \chi^{-k}$ and $\sum_{k=0}^{\infty}  \chi^{-k}$ converge, therefore after iterating the above inequality we have
	$$
	\left(\int_{I_{n+1}}|\overline{w}|^{s_{n+1}} dx\right)^{\frac{1}{s_{n+1}}}\leq C \left(\int_0^1|\overline{w}|^{2} dx\right)^{\frac{1}{2}},
	$$
	for some constant \(C\) independent of \(n\). Finally, by passing to the limit $n\to \infty$ we have
	\begin{equation}
		\|w\|_{L^{\infty}(0, \frac{1}{2})} \leq C \pt{\norm{w}_{L^2}+k},
	\end{equation}
	and since we already know that \(\norm{w}_{L^2}\leq \norm{w}_{X^\alpha}\leq C\norm{g}_{L^s}\) we deduce
	$$
	\left\| w\right\|_{L^\infty(0,\frac12)} \leq C \left\| g\right\|_{L^s(0,1)}
	$$
	and the proof is completed if we notice that we already know that
	$$
	\left\| w\right\|_{L^\infty(\frac14,1)} \leq C\left\| g\right\|_{L^s(0,1)}
	$$
	from the fact that \(X^{\alpha}(0,1)\hookrightarrow H^1(\frac14,1)\hookrightarrow L^\infty(\frac14,1)\).
\end{proof}

\section{Analysis of \(u_D\in X^\alpha_{00}\)}\label{reg-Dir}

As we established in \cref{sect-existence,sect-unique} for any \(\alpha<\frac{1}{2}\) and $f\in L^p(0,1)$ with $1\leq p\leq \infty$ there exists a unique $u_D\in X_{00}^{\alpha}$ satisfying \eqref{SL-eq} such that
\[
\norm{u_D}_{X^{\alpha}}=\norm{\varphi_f}_{\pt{X^\alpha}^*}\leq C\norm{f}_{L^p}.
\]
Throughout the rest of this section and to ease the notation we will denote by \(u\) the solution with the above properties.

\subsection{Regularity}
Firstly observe that because $X^{\alpha}\hookrightarrow C[0,1]\subseteq L^p(0,1)$ for \(\alpha<\frac12\) we also have $\norm{u}_{L^p}\leq C\norm{f}_{L^p}$, and what remains to be proven is that \(u\in W^{2,p}_{loc}((0,1])\), \(u\) satisfies \eqref{SL-eq} a.e. and the following properties:
\begin{enumerate}[label=(\roman*)]
	\item\label{tag1}  $x^{2\alpha}u'\in W^{1,p}(0,1)$ with $\|x^{2\alpha}u'\|_{W^{1,p}}\leq C \|f\|_{L^p},$
	\item\label{tag2} $x^{2\alpha-1}u \in W^{1,p}$ with $\|x^{2\alpha-1}u\|_{W^{1,p}}\leq C \|f\|_{L^p},$
	\item\label{tag3} $x^{2\alpha}u\in W^{2,p}(0,1)$ with $\|x^{2\alpha}u\|_{W^{2,p}}\leq C \|f\|_{L^p}.$
	\item\label{tag4} $u\in C^{0,\frac{1}{2}-\alpha}[0,1],$ with $\|u\|_{C^{0,\frac{1}{2}-\alpha}}\leq C \|f\|_{L^p}$ when \(0<\alpha<\frac12\).
	\item\label{tag5} $\lim\limits_{x\to 0^+} x^{2\alpha-1+\frac1p}u(x)=0,$ for $2\alpha+\frac1p<1$ and \(1\leq p\leq\infty\).
\end{enumerate}

To do the above, firstly observe that by taking $v\in C_c^{\infty}(0,1)\subseteq X^\alpha_{00}$ in \eqref{w-SL-eq} we have that $w'(x)=u(x)-f(x)$ a.e. where $w(x)=x^{2\alpha}u'(x)$, that is \eqref{SL-eq} is satisfied a.e. in (0,1). Since we already know that \(u\in L^p\), the previous observation tells us that 
\[
\norm{w'}_{L^p}\leq \norm{u}_{L^p}+\norm{f}_{L^p}\leq C\norm{f}_{L^p}.
\]

In order to continue, we divide the analysis into two cases: \(0<\alpha<\frac12\) and \(\alpha<0\):

\subsubsection{The case \(0<\alpha<\frac12\)} Since \(\alpha>0\) and \(u\in X^\alpha\) we deduce that \(w=x^\alpha\cdot x^\alpha u'\) satisfies
\[
\norm{w}_{L^1}\leq \|w\|_{L^2}\leq \|x^{\alpha}u'\|_{L^2}\leq \|u\|_{X^\alpha}\leq C\|f\|_{L^p},
\]
and since $w'\in L^p(0,1)$ we conclude that $w\in W^{1,1}(0,1)$ with
\[
\norm{w}_{W^{1,1}}\leq C\norm{f}_{L^p}.
\]
But because $W^{1,1}(0,1)\hookrightarrow L^p(0,1)$ continuously, we deduce that $\|w\|_{L^p}\leq C\|f\| _{L^p}$ and as a consequence \(w\in W^{1,p}(0,1)\) with $\|w\|_{W^{1,p}}\leq C\|f\|_{L^p}$ which proves \cref{tag1}.

Additionally, because \(w\in W^{1,p}(0,1)\hookrightarrow C([0,1])\) and because \(\alpha<\frac12\) we have
\[
\lim_{s\to 0^+} su'(s)= \lim_{s\to 0^+} s^{1-2\alpha}w(s)=0,
\]
from where we can write the identity
\[
x^{2\alpha-1}u(x)=\frac{x^{2\alpha}u'(x)}{1-2\alpha }+\frac{x^{2\alpha-1}}{2\alpha-1}\int_0^x(s^{2\alpha}u'(s))'s^{-2\alpha+1}ds\qquad\forall\,x\in(0,1),
\]
which gives
\begin{equation}\label{us-iden-1}
	(x^{2\alpha-1}u(x))'=x^{2\alpha-2}\int_0^xw'(s)s^{1-2\alpha}ds.
\end{equation}

If $p=1$ we can use Tonelli's theorem to write
\begin{align*}
\int_0^1|(x^{2\alpha-1}u(x))'|dx&\leq  \int_0^1x^{2\alpha-2}\int_0^x|w'(s)|s^{1-2\alpha}ds\, dx\\
	&=\int_0^1 |w'(s)|s^{1-2\alpha} \int_s^1 x^{2\alpha-2}dx \, ds\\
	&=\frac{1}{1-2\alpha}\int_0^1 |w'(s)| s^{1-2\alpha}(s^{2\alpha-1}-1) ds\\
	& =\frac{1}{1-2\alpha}\int_0^1 |w'(s)| (1-s^{1-2\alpha}) ds\\
	& \leq C \|w'\|_{L^1}\\
	&\leq C \|f\|_{L^1}.
\end{align*}
If $1<p<\infty$ we use Hardy's inequality to get
\begin{align*}
\int_0^1|(x^{2\alpha-1}u(x))'|^p dx&\leq \int_0^1\pt{x^{2\alpha-2}\int_0^x|w'(s)|s^{1-2\alpha}ds}^p dx\\
&\leq \int_0^1\left(\frac{1}{x}\int_0^x|w'(s)|ds\right)^pdx\\
&\leq C\|w'\|_{L^p}^p\\
&\leq C\norm{f}_{L^p}^p.
\end{align*}
And if \(p=\infty\) then from \eqref{us-iden-1} we obtain
\[
\abs{(x^{2\alpha-1}u(x))'}\leq x^{2\alpha-2}\int_0^x\abs{w'(s)}s^{1-2\alpha}ds\leq \frac{1}{2-2\alpha}\norm{w'}_\infty\leq C\norm{f}_{L^\infty},
\]
so that in every possible case we deduce that $\|(x^{2\alpha-1}u)'\|_{L^p}\leq C\|f\|_{L^p}$. Furthermore, we have that
\[
(x^{2\alpha-1}u(x))'=(2\alpha-1)x^{2\alpha-2}u(x)+x^{2\alpha-1}u'(x),
\]
where if we notice that $x^{2\alpha-1}u(x)=\frac{1}{2\alpha-1}\left(x(x^{2\alpha-1}u(x))' - w(x)\right)$ then it follows that
\[
\|x^{2\alpha-1}u\|_{L^p}\leq C\|f\|_{L^p},
\]
so that \(x^{2\alpha-1}u\in W^{1,p}(0,1)\) with $\|x^{2\alpha-1}u\|_{W^{1,p}}\leq C\|f\|_{L^p}$ and \cref{tag2} is proven.

Observe that by writing \(x^{2\alpha}u=x\cdot x^{2\alpha-1}u\) we see that
\[
\|x^{2\alpha}u\|_{L^p}\leq\|x^{2\alpha-1}u\|_{L^p}\leq C\|f\|_{L^p},
\]
and also we have 
\[
(x^{2\alpha}u(x))'=2\alpha x^{2\alpha-1}u(x)+w(x)\in W^{1,p}(0,1),
\]
thus $\|x^{2\alpha}u\|_{W^{2,p}}\leq C\|f\|_{L^p}$ and \cref{tag3} is proven and as a consequence we also have that \(u\in W^{2,p}_{loc}((0,1])\).

Finally \cref{tag4} follows from the embedding \(X^\alpha\hookrightarrow C^{\frac12-\alpha}([0,1])\) whereas to obtain \cref{tag5} when \(0<\alpha<\frac12\), we recall that \(u(0)=0\) so we can we write
$$
|u(x)|\leq \int_0^x|u'(s)|ds=\int_0^x |w(s)|s^{-2\alpha}ds
$$
and we observe that for $1<p<\infty$ and \(2\alpha+\frac1p<1\) this gives
$$
|u(x)|\leq C\left(\int_0^x|w(s)|^{p}ds\right)^{1/p}x^{1-2\alpha-\frac{1}{p}}
$$
where as if \(p=\infty\) one gets
$$
|u(x)|\leq C\norm{w}_{L^\infty(0,x)}x^{1-2\alpha}
$$
so that if \(2\alpha+\frac1p<1\) we improve the behavior at \(x=0\) by
\[
\abs{u(x)}\leq Cx^{1-2\alpha-\frac1p}\norm{f}_p\quad\text{if }1<p\leq\infty.
\]

\subsubsection{The case \(\alpha<0\)}

To prove that \(w=x^{2\alpha}u'\in W^{1,p}(0,1)\) we need to proceed differently as in the case \(\alpha>0\) because we do not longer have the inequality \(\norm{w}_{L^2}\leq \norm{x^\alpha u'}_{L^2}\) directly. Instead we take $g\in C^2(0,1]\cap C([0,1])$ as the solution of \eqref{Aux}, namely of
\begin{equation}
\left\{
\begin{aligned}
	-(x^{2\alpha}g'(x))'+g(x)&=0 & \text{ in } (0,1],\\
	g(1)&= 0, && \\
	\lim_{x\to 0^+}g(x)&= 1. &&
\end{aligned}
\right.
\end{equation}
whose existence and uniqueness is guaranteed by \cref{lemmabes3} in the Appendix. For \(\ve\in (0,1)\) we multiply \eqref{SL-eq} by \(g\) and integrate by parts on the interval \((\ve,1)\) to obtain
$$
-w(x)g(x)|_{\ve}^1 +\int_{\ve}^1x^{2\alpha}u'(x)g'(x)dx+\int_{\ve}^1u(x)g(x)dx= \int_{\ve}^1f(x)g(x)dx.
$$
Since \(g\in C^2((0,1])\) and $g(1)=0$ we can integrate by parts once more to obtain
\begin{multline*}
	w(\ve)g(\ve)+x^{2\alpha}g'(x)u(x)\Big|_{\ve}^1 -\int_{\ve}^1(x^{2\alpha}g'(x))'u(x)dx+\int_{\ve}^1g(x)u(x)dx\\ =\int_{\ve}^1f(x)g(x)dx.
\end{multline*}

Because $\lim\limits_{x\to 0^+}u(x)=u(1)=0$, $\lim\limits_{x\to 0^+} x^{2\alpha} g'(x)$ converges, and since $g$ satisfies (\ref{Aux}) we have as a consequence that $\lim\limits_{x\to 0^+}w(x)$ exists and it is given by
\begin{equation}\label{limit-w}
\lim_{x\to 0^+}w(x)= \int_0^1f(x)g(x)dx,
\end{equation}
furthermore, for $x \in(0,1)$ it holds that
$$
w(x)= \int_0^x w'(s)ds + \lim_{s\to 0^+}w(s),
$$
so that
\[
|w(x)|\leq \int_0^1 |w'(s)|ds + \int_0^1|f(x)g(x)|dx
\]
and as a consequence
$$
|w(x)|\leq \int_0^1 |w'(s)|ds + \int_0^1|f(x)g(x)|dx ,
$$
but then
$$
\norm{w}_{L^p}\leq \|w\|_{L^{\infty}}\leq \|w'\|_{L^1}+ \|g\|_{L^\infty}\, \|f\|_{L^1}\leq C\pt{\|w'\|_{L^p}+\|f\|_{L^p}}\leq C\norm{f}_{L^p},
$$
thus we have shown that $\|w\|_{W^{1,p}}\leq C\|f\|_{L^p}$.

Notice that \eqref{us-iden-1} remains valid for \(\alpha<0\) and hence 
$$
|(x^{2\alpha-1}u(x))'|\leq x^{2\alpha-2}\int_0^x|w'(s)|s^{1-2\alpha}ds.
$$
therefore the same conclusions reached for the case \(0<\alpha<\frac12\) can be reached for \(\alpha<0\), with the addition that in this case the limit
\[
\lim_{x\to 0^+}x^{2\alpha-1+\frac1p}u(x)=0\quad\text{if }1<p<\infty\\
\]
also works for \(p=1\). We omit those details.

\subsection{Optimality of the behavior near \(0\)}\label{pf-rem1}
Here we prove \cref{obsnueva}, that is, we analyze the following quantity
$$
K_D(x)=\sup_{\|f\|_{L^p}\leq1}\left|x^{2\alpha-1}u(x)\right|
$$
and we show it is bounded above and below for sufficiently small \(x\). For that purpose consider $g\in C^2((0,1])\cap C([0,1])$, the solution of \eqref{Aux} and use \eqref{limit-w} to write
$$
t^{2\alpha}u'(t)=\int_0^1 g(s)f(s)ds+\int_0^t(s^{2\alpha}u'(s))'ds\qquad\forall\,t\in (0,1),
$$
thus
\begin{align*}
	u(x)&=\int_0^x t^{-2\alpha}\int_0^1 g(s)f(s)ds\,dt,+\int_0^xt^{-2\alpha}\int_0^t(f(s)-u(s))ds\,dt\\
	&=\int_0^x t^{-2\alpha}\int_0^1 g(s)f(s)ds\,dt,+\int_0^x\int_t^x t^{-2\alpha}(f(s)-u(s))dt\,ds\\
	&= \frac{x^{1-2\alpha}}{1-2\alpha}\int_0^1 g(s)f(s)ds
	+\frac{1}{1-2\alpha}\int_{0}^{x}(f(s)-u(s))(x^{1-2\alpha}-s^{1-2\alpha})ds
\end{align*}
which gives
$$
|x^{2\alpha-1}u(x)|\leq C\pt{\int_0^1|f(s)g(s)|ds+\norm{f}_{L^p}},
$$
because \(\norm{f-u}_{L^1}\leq \norm{f}_{L^p}+\norm{u}_{L^p}\leq (1+C)\norm{f}_{L^p}\) and \(s^{1-2\alpha}\leq x^{1-2\alpha}\) for \(s\leq x\) and \(\alpha<\frac12\). In particular we obtain that
\[
K_D(x)\leq C(1+\|g\|_{\infty}).
\]

On the other hand, observe that for \(1<p\leq \infty\) we have
$$
\abs{x^{2\alpha-1}\int_{0}^{x}s^{1-2\alpha}(f(s)-u(s))ds}\leq \int_0^x\abs{f(s)-u(s)}ds=x^{1-\frac1p}\norm{f}_p,
$$
and for \(p=1\)
$$
\abs{x^{2\alpha-1}\int_{0}^{x}s^{1-2\alpha}(f(s)-u(s))ds}\leq \int_0^x\abs{f(s)-u(s)}ds=o(x),
$$
where \(o(x)\) is a quantity that goes to \(0\) as \(x\to 0\) (depending on \(f\)). Therefore for every \(1\leq p\leq \infty\) we have
$$
x^{2\alpha-1}u(x) =\frac{1}{1-2\alpha}\int_0^1g(s)f(s)ds+o(x).
$$
and if we fix \(f\equiv 1\) we get 
\[
K_D(x)\geq \frac{1}{1-2\alpha}\abs{\int_0^1g(s)ds}-o(x)>0
\]
for all sufficiently small \(x>0\)

\section{Analysis of \(u_N\in X^\alpha_{\cdot 0}\)}\label{reg-Neu}

\subsection{The cases \(\alpha<\frac12\), or \(\frac12\leq \alpha<1\) and \(p\geq \frac{2}{3-2\alpha}\)}\label{sect-pbig}

As we established in \cref{sect-existence} there exists $u_N\in X_{\cdot 0}^\alpha$ satisfying \eqref{w-SL-eq} for every \(v\in X^\alpha_{\cdot 0}\) with $\|u_N\|_{X^{\alpha}_{\cdot 0}}\leq C\|f\|_{L^p}$. Since no confusion is present, we will simple denote by \(u\) such solution throughout this part. Also we will suppose for the moment that \(u\in L^p\) with \(\norm{u}_p\leq C\norm{f}_p\) to do some computations and later we will prove it.

By taking $v\in C_c^{\infty}(0,1)$, then $w'= u-f$ a.e. in (0,1) for $w(x)= x^{2\alpha}u'(x)$ from where we obtain 
$$
\|w'\|_{L^p}\leq \|u\|_{L^p}+ \|f\|_{L^p}\leq C \|f\|_{L^p}.
$$
For \(0<\ve<1\) we multiply \eqref{SL-eq} by $v\in C^2[0,1]$ and integrate over \((\ve,1)\) to obtain
\begin{align*}
0&=-\int_{\ve}^1 (x^{2\alpha} u'(x))' v(x) dx + \int_{\ve}^1 (u(x)-f(x))v(x)dx\\
	&=-x^{2\alpha} u'(x)v(x)\Big|_{\ve}^1 + \int_{\ve}^1 x^{2\alpha}u'(x)v'(x)dx+\int_{\ve}^1 u(x)v(x)dx-\int_{\ve}^1 f(x)v(x)dx,
\end{align*}
therefore, $\lim_{\ve\to 0^+} x^{2\alpha} u'(x)v(x)\Big|_{\ve}^1=0$ and if we take $v$ such that $v(1)=0$ and $v(0)=1$ then we have that
\begin{equation}\label{*3}
	\lim_{x\to 0^+}x^{2\alpha}u'(x)=0.
\end{equation}
The above allows us to write for every \(x\in (0,1)\)
$$
|w(x)|\leq \int_0^x |(x^{2\alpha}u'(x))'|dx\leq \|w'\|_{L^p}\leq C\|f\|_{L^p},
$$
which implies that, $\|w\|_{W^{1,p}(0,1)}\leq C\|f\|_{L^p}$. In particular $\|u'\|_{L^p} \leq C_{\alpha,p}\|f\|_{L^p}$ when \(\alpha<0\), so that $w,u\in W^{1,p}(0,1)$ in this case.

Now for $1\leq p<\infty$ we have
$$
|x^{2\alpha}u'(x)|\leq \int_0^x |w'(x)|dx\leq x^{1-\frac{1}{p}}\left(\int_0^x |w'(x)|^pdx\right)^{\frac{1}{p}},
$$
and if \(p=\infty\)
$$
|x^{2\alpha}u'(x)|\leq x\norm{w'}_{L^\infty},
$$
so it follows that
\[
\lim_{x\to 0^+}x^{2\alpha-1+\frac{1}{p}}u'(x)=0\quad\text{for every }1< p\leq \infty.
\]
Observe that the case \(p=1\) can be included in the above limit because
\[
x^{2\alpha}u'(x)\xrightarrow[x\to 0^+]{}0
\]
thanks to \eqref{*3}. Similarly we have 
$$
|x^{2\alpha-1}u'|\leq\frac{1}{x}\int_0^x|w'(s)|ds,
$$
so that for \(1<p\leq \infty\) we can use Hardy's inequality to obtain that
\begin{equation}\label{*4}
\begin{aligned}
	\int_0^1|x^{2\alpha-1}u'(x)|^pdx &\leq \int_0^1 \left(\frac{1}{x}\int_0^x\abs{w'(s)}ds\right)^pdx\\
	&\leq C \|w'\|_{L^p}^p,
\end{aligned}
\end{equation}
therefore $\|x^{2\alpha-1}u'\|_{L^p}\leq C\|f\|_{L^p}$ and since we are assuming that  \(\norm{u}_{L^p}\leq C\norm{f}_{L^p}\), from the equation satisfied by \(u\) we also conclude that \(\norm{x^{2\alpha} u''}_{L^p}\leq C\norm{f}_{L^p}\). Observe that all these bounds tell us that \(u\in W^{2,p}_{loc}(0,1)\).

To conclude we need to show that \(u\in L^p(0,1)\) with \(\norm{u}_p\leq C\norm{f}_p\). Firstly observe that if \(\alpha<\frac12\) then we have the embedding $X^\alpha\hookrightarrow L^\infty(0,1)$ so the fact that \(u\in X^\alpha_{\cdot 0}\) readily implies that $\|u\|_{L^p}\leq C\|f\|_{L^p}$. Secondly, if \(\alpha=\frac12\) and \(p<\infty\) then we have the embedding $X^\alpha\hookrightarrow L^q(0,1)$ for every \(q<\infty\), in particular for \(q=p\) it follows that $\|u\|_{L^p}\leq C\|f\|_{L^p}$, just as before. Therefore we are left with: the case \(\alpha=\frac12\) with \(p=\infty\), and the case \(\frac{1}{2}<\alpha<1\) with \(1\leq p\leq \infty\).

\begin{remark}\label{rem-p0-1}
Observe that for every \(\alpha<\frac12\) and every \(1\leq p<\infty\) the above shows that there exists some \(p_0>p\) such that in fact \(\norm{u}_{L^{p_0}}\leq C\norm{f}_{L^p}\).
\end{remark}

\subsubsection{The case \(\alpha=\frac12\) and \(p=\infty\)}

Recall that \(X^{\alpha}\hookrightarrow L^2\) so in particular we have that
\[
\norm{u}_{L^2}\leq \norm{f}_{L^\infty}.
\]
By following what we did to reach \eqref{*4} for \(\alpha=\frac12\) we deduce that
\[
\norm{u'}_{L^2}\leq C\norm{(xu')'}_{L^2}=C\norm{u-f}_{L^2}\leq C\norm{f}_{L^\infty},
\]
so that \(u\in H^1_0(0,1)\), which in particular implies
\[
\norm{u}_{L^\infty}\leq C\norm{f}_{L^\infty}
\]
due to the embedding \(H^1(0,1)\hookrightarrow L^\infty(0,1)\).

\subsubsection{The case \(\frac12<\alpha<1\) and \(\frac{2}{3-2\alpha}\leq p\leq \infty\)}\label{ssect-iter}

Notice that if $\frac{1}{2}<\alpha<1$ then it holds that $\frac{2}{3-2\alpha}< \frac{2}{2\alpha-1}$, and we analyze the situation by cases.

If $\frac{2}{3-2\alpha}\leq p\leq \frac{2}{2\alpha-1}$, and because $X^{\alpha} \hookrightarrow L^q(0,1)$ for all $q\leq \frac{2}{2\alpha -1}$, it follows that
$$
\|u\|_{L^p} \leq C \|u\|_{X^{\alpha}}  \leq C \|f\|_{L^p},
$$
and we are done.

If $\frac{2}{2\alpha-1}<p$: by the same argument used in \eqref{*4} we have that
\begin{align*}
	\|x^{2\alpha-1}u'\|_{L^{\frac{2}{2\alpha-1}}}&\leq C\norm{w'}_{L^{\frac{2}{2\alpha-1}}}\\
	&\leq C\left(\|u\|_{\frac{2}{2\alpha-1}}+\|f\|_{\frac{2}{2\alpha-1}}\right)\\
	&\leq C\left(\|u\|_{X^{\alpha}}+\|f\|_{L^p}\right)\\
	&\leq C\|f\|_{L^p},
\end{align*}
hence $x^{2\alpha-1}u' \in L^{\frac{2}{2\alpha-1}}$ and $u \in L^{\frac{2}{2\alpha -1}}$, therefore $u\in X^{2\alpha-1, {\frac{2}{2\alpha-1}}}$ with
\[
\|u\|_ {X^{2\alpha-1, {\frac{2}{2\alpha-1}}}}\leq C \|f\|_{L^p}.
\]
Furthermore, from \cite[Theorem A.2]{CW11-1} we have
\begin{equation*} 
	u\in X^{2\alpha-1, {\frac{2}{2\alpha-1}}} \hookrightarrow
	\begin{dcases}
		L^q(0,1)& \text{ for all } q<\infty  \quad\text{ if } \alpha \leq\frac{5}{6}\\
		L^{q}(0,1)& \text{ for all } q\leq\frac{2}{6\alpha-5} \quad\text{ if } \frac{5}{6}<\alpha< 1.
	\end{dcases}
\end{equation*}
Therefore for $\alpha<\frac{5}{6}$ it follows that $\|u\|_{L^p} \leq C\|u\|_{ X^{2\alpha-1, {\frac{2}{2\alpha-1}}}}\leq C \|f\|_{L^p}$, whereas for $\frac{5}{6}<\alpha<1$ we have that $\frac{2}{2\alpha-1}<\frac{2}{6\alpha-5}$.

We repeat the above argument and separate the analysis into two cases: if $\frac{2}{2\alpha-1}\leq p\leq\frac{2}{6\alpha-5}$: as $X^{2\alpha-1, {\frac{2 }{2\alpha-1}}} \hookrightarrow L^{q}(0,1) \text{ for all } q \leq \frac{2}{6\alpha-5}$, then it follows that
$$
\|u\|_{L^p} \leq C_{\alpha} \|u\|_{X^{2\alpha-1, {\frac{2}{2\alpha-1}}} }  \leq C \|f\|_{L^p}.
$$
And if $\frac{2}{6\alpha-5}<p$: from (\ref{*4}) we have
\begin{align*}
	\|x^{2\alpha-1}u'\|_{L^{\frac{2}{6\alpha-5}}}&\leq C_{p} \left(\|u\|_{\frac{2}{6\alpha-5}}+\|f\|_{\frac{2}{6\alpha-5}}\right)\\
	&\leq C\left(\|u\|_{X^{2\alpha-1, {\frac{2}{2\alpha-1}}}}+\|f\|_{L^p}\right)\\
	&\leq C\|f\|_{L^p},
\end{align*}
hence
\begin{equation*} 
	u\in X^{2\alpha-1, {\frac{2}{6\alpha-5}}} \hookrightarrow
	\begin{dcases}
		L^q(0,1)& \text{ for all } q<\infty  \quad \text{ if } \alpha \leq\frac{9}{10},\\
		 L^{q}(0,1)& \text{ for all } q  \leq \frac{2}{10\alpha-9} \quad  \text{ if } \frac{9}{10}<\alpha< 1,
	\end{dcases}
\end{equation*}
therefore if $\alpha<\frac{9}{10}$ then $\|u\|_{L^p}\leq C \|f\|_{L^p}$.

The above can be repeated inductively in the following fashion. Let $n\in \NN$ and suppose that for $\alpha > \frac{4n-3}{4n-2}$ we have
\begin{equation*} 
	u\in X^{2\alpha-1, r_{n-1}} \hookrightarrow
	\begin{dcases}
		L^q(0,1)& \text{ for all } q<\infty  \quad \text{ if } \alpha \leq \frac{4n+1}{4n+2}\\  L^{q}(0,1)& \text{ for all } q  \leq r_n \quad  \text{ if } \frac{4n+1}{4n+2}<\alpha< 1
	\end{dcases}
\end{equation*}
with $\|u\|_{ \mathbb{Y}}\leq C \|f\|_{L^p}$, where $\mathbb{Y}:=  X^{2\alpha-1,r_{n-1} }$ and \(r_n=\frac{2}{(4n+2)\alpha-(4n+1)}\).

Since $\mathbb{Y} \hookrightarrow L^q(0,1)$ for all $q \leq r_n$ and if $r_{n-1}\leq p\leq r_n$ then we have that $\|u\|_{L^p}\leq C\|f\| _{L^p}$. Whereas if $r_{n-1}<r_n<p $ we have 
\begin{align*}
	\|x^{2\alpha-1}u'\|_{L^{r_n}}&\leq C_{p} \left(\|u\|_{L^{r_n}}+\|f\|_{L^{r_n}}\right)\\
	&\leq C_{\alpha}\left(\|u\|_{\mathbb{Y}}+\|f\|_{L^p}\right)\\
	&\leq C\|f\|_{L^p}.
\end{align*}
Therefore
\begin{equation*} 
	u\in X^{2\alpha-1, r_n} \hookrightarrow
	\begin{dcases}
		L^q(0,1)& \text{ for all } q<\infty  \quad \text{ if } \alpha \leq \frac{4(n+1)+1}{4(n+1)+2}\\
		L^{q}(0,1)& \text{ for all } q  \leq r_{n+1} \quad  \text{ if } \frac{4(n+1)+1}{4n+2}<\alpha< 1
	\end{dcases}
\end{equation*}
with $\|u\|_{X^{{2\alpha}-1, r_n}} \leq C\|f\|_{L^p} $.

And since for any $\alpha<1$ there exists $n\in \mathbb{N}$ such that $\alpha < \frac{4n+1 }{4n+2}$ we can conclude that $\|u\|_{L^p}\leq C\|f\|_{L^p}$ for every \(\alpha<1\) and every \(\frac{2}{2\alpha-1}\leq p\leq \infty\).

\begin{remark}\label{rem-p0-2}
Note that for every \(\frac12\leq\alpha<1\) and every \(\frac2{3-2\alpha}\leq p<\infty\) the above shows that there exists some \(p_0>p\) for which \(\norm{u}_{L^{p_0}}\leq C\norm{f}_{L^p}\).
\end{remark}

\subsection{The case \(\frac12<\alpha<1\) and \(1\leq p<\frac{2}{3-2\alpha}\)}\label{ssect-iter2}
We just need to prove that \(\norm{u}_{L^p}\leq C\norm{f}_{L^p}\) as the rest of the argument from \cref{sect-pbig} remains valid if \(u\in X^{\alpha,q}_{\cdot 0}\) instead of in \(X^{\alpha}_{\cdot 0}\). The argument to show the \(L^p\) bound for \(u\) is similar to what we did in \cref{ssect-iter}, so we only show the first steps and the rest is left to the reader.

Recall that the solution found belongs to \(X^{\alpha,q}_{\cdot 0}\) for all \(1\leq q<\frac1\alpha\leq 2\). We identify two cases, either \(\alpha\leq\frac34 \Leftrightarrow \frac{2}{3-2\alpha}\leq\frac1\alpha\) or \(\frac34<\alpha<1\Leftrightarrow\frac{2}{3-2\alpha}\geq \frac1\alpha\). In the first case, since \(u\in X^{\alpha,q}_{\cdot 0}\) for every \(q<\frac1\alpha\), then in particular we can take \(q=p<\frac{2}{3-2\alpha}\leq\frac1\alpha\) so that \(u\in X^{\alpha,p}_{\cdot 0}\) and therefore
\begin{equation*}\label{siter1}
\alpha\leq \frac34\Rightarrow\norm{u}_{L^p}\leq C\norm{f}_{L^p}.
\end{equation*}

If \(\frac34< \alpha<1\Leftrightarrow\frac{2}{3-2\alpha}>\frac1\alpha\) we also have two cases, either \(1\leq p<\frac1\alpha\) or \(\frac1\alpha\leq p<\frac2{3-2\alpha}\). If \(1\leq p<\frac1\alpha\) then we again can use \(q=p\) and the solution belongs to \(X^{\alpha,p}_{\cdot 0}\) so in particular 
\[
\norm{u}_{L^p}\leq C\norm{f}_{L^p}.
\]

If \(\frac1\alpha\leq p<\frac2{3-2\alpha}\) we observe that the solution \(u\) still verifies the weak equation
\[
\int_0^1\pt{x^{2\alpha}u'(x)v'(x)+u(x)v(x)}dx=\int_0^1f(x)v(x)dx\qquad \forall\,v\in C_c^1([0,1]),
\]
and in particular we still have that \(-(x^{2\alpha}u')'+u(x)=f(x)\) a.e. with 
\[
\lim_{x\to 0^+}x^{2\alpha}u'(x)=0,
\]
so that the identity 
\[
x^{2\alpha}u'(x)=\int_0^x (u(s)-f(s))ds\Rightarrow x^{2\alpha-1}u'(x)=\frac1x\int_0^x (u(s)-f(s))ds
\]
remains valid. In particular, for any \(q>1\) Hardy's inequality tells us that
\[
\norm{x^{2\alpha-1}u'}_{L^q}\leq C\pt{\norm{u}_{L^q}+\norm{f}_{L^q}}.
\]
Hence if \(\frac1\alpha\leq p<\frac2{3-2\alpha}\) then we can use any \(1<q<\frac1\alpha\) and write
\[
\norm{x^{2\alpha-1}u'}_{L^q}\leq C\pt{\norm{u}_{L^q}+\norm{f}_{L^q}}\leq C\norm{f}_{L^p}.
\]
In particular \(u\in X^{2\alpha-1,q}\) for any \(1<q<\frac1\alpha\) and because \(\alpha\geq\frac34\) we have \(q<\frac1\alpha<\frac{1}{2(1-\alpha)}\) and as a consequence
\[
\norm{u}_{L^t}\leq C\norm{u}_{X^{2\alpha-1,q}}
\]
for all \(t<\frac{1}{3\alpha-2}\).

The above can be divided again into two cases, either \(\alpha\leq \frac78\Leftrightarrow \frac{1}{3\alpha-2}\geq \frac{2}{3-2\alpha}\) or \(\frac78<\alpha<1\Leftrightarrow\frac{1}{3\alpha-2}<\frac{2}{3-2\alpha}\). If \(\frac34\leq \alpha\leq \frac78\) we have \(p<\frac{2}{3-2\alpha}\leq \frac{1}{3\alpha-2}\) so we can take \(t=p\) and deduce that
\begin{equation*}\label{siter2}
\alpha\leq \frac78\Rightarrow \norm{u}_{L^p}\leq C\norm{f}_{L^p},
\end{equation*}
whereas if \(\frac78<\alpha<1\Leftrightarrow\frac{1}{3\alpha-2}<\frac{2}{3-2\alpha}\) we consider the two cases: \(1\leq p<\frac1{3\alpha-2}\) and \(\frac1{3\alpha-2}\leq p<\frac{2}{3-2\alpha}\). Just as before, if \(1\leq p<\frac1{3\alpha-2}\) then we immediately deduce that
\[
\norm{u}_{L^p}\leq C\norm{f}_{L^p},
\]
and if \(\frac1{3\alpha-2}\leq p<\frac{2}{3-2\alpha}\) then \(u\in X^{2\alpha-1,q}\) for any \(q<\frac1{3\alpha-2}\). Observe that if \(\frac78<\alpha\leq \frac45\) then \(X^{2\alpha-1,q}\) embeds into \(L^t\) for all \(t<\infty\) hence
\[
\norm{u}_{L^p}\leq C\norm{f}_{L^p}
\]
and if \(\frac45<\alpha<1\) then 
\[
\norm{u}_{L^t}\leq C\norm{f}_{L^p}\qquad\forall\, t<\frac{1}{5\alpha-4}.
\]
In particular
\begin{equation*}\label{siter3}
\pt{\alpha\leq \frac{11}{12}\Leftrightarrow \frac{2}{3-2\alpha}\leq \frac{1}{5\alpha-4}}\Rightarrow \norm{u}_{L^p}\leq C\norm{f}_{L^p}.
\end{equation*}

As the reader can easily verify, after \(n\) steps of the above iteration we are led to the following: for any \(n\in \NN\) we have
\[
\alpha\leq \frac{3+4n}{4+4n}\Rightarrow \norm{u}_{L^p}\leq C\norm{f}_{L^p},
\]
and the desired bound is obtained since \(\frac{3+4n}{4+4n}\xrightarrow[n\to\infty]{}1\).

\begin{remark}\label{rem-p0-3}
Notice again that for every \(\frac12\leq\alpha<1\) and every \(1\leq p<\frac2{3-2\alpha}\) there exists some \(p_0>p\) such that \(\norm{u}_{L^{p_0}}\leq C\norm{f}_{L^p}\).
\end{remark}

\subsubsection{The proof of \cref{rem-p=12}}\label{rem-p=13}

As me mentioned in \cref{rem-p=12}, and as the reader can see from the above results, the existence part in \cref{Thm-NEU} and the fact that \(\norm{u_N}_{L^p}\leq C\norm{f}_{L^p}\) remain valid for \(p=1\), but it is not necessarily true that
\[
x^{2\alpha-1}u'\in L^1(0,1)
\]
for given \(f\in L^1(0,1)\). To see this consider \(f(x)=\pt{x(1-\ln x)^{\frac32}}^{-1}\) and observe that it belongs to \(L^1(0,1)\) so but we did before we can find a solution \(u_N\) in \(X^{\alpha,q}\) for every \(1\leq q<\frac1\alpha\). In particular \(u_N\in L^q\) for some \(q>1\). However if we consider the problem
\[
\left\{
\begin{aligned}
-(x^{2\alpha}v'(x))'&=f(x)&&\text{in }(0,1)\\
v(1)&=0\\
\lim_{x\to 0^+}x^{2\alpha}v'(x)&=0
\end{aligned}
\right.
\]
then a direct computation tells us that its solution verifies
\[
x^{2\alpha}v'(x)=-\frac{2}{(1-\ln x)^{\frac12}}.
\]
Finally, from \eqref{SL-eq} and the fact that \(x^{2\alpha}u_N'(x)\xrightarrow[x\to 0^+]{}0 \) we can write
\begin{align*}
x^{2\alpha-1}u_N'(x)&=\frac1x\int_0^x(u(s)-f(s))ds\\
&=\frac1x\int_0^x u_N(s)ds+x^{2\alpha-1}v'(x)\\
&=\frac1x\int_0^x u_N(s)ds-\frac{2}{x(1-\ln x)^{\frac12}}
\end{align*}
and thanks to Hardy's inequality the first term belongs to \(L^q(0,1)\subseteq L^1(0,1)\), but the second term is not in \(L^1(0,1)\), therefore \(x^{2\alpha-1}u_N'\) cannot be in \(L^1\) and \emph{a fortiori} neither can \(x^{2\alpha}u_N''\).

\subsection{Optimality of the behavior near \(0\)}\label{pf-rem2}

Here we show \cref{obsN2}, that is we prove that
\[
K_N(x)=\sup_{\|f\|_{L^p}\leq1}\left|x^{2\alpha-1+\frac{1}{p}}u'(x)\right|.
\]
is bounded above and below for sufficiently small \(x\), to do we notice that for \(1<p<\infty\)
\begin{align*}
	x^{2\alpha-1+\frac{1}{p}}u'(x) = x^{\frac{1}{p} -1}\int_0^x
	(s^{2\alpha}u'(s))' ds = x^{\frac{1}{p} -1}\int_0^x (u(s)- f(s))ds.
\end{align*}
Then,
\begin{align*}
	|x^{2\alpha-1+\frac{1}{p}}u'(x)|&\leq x^{\frac{1}{p} -1}\int_0^x |f(s)-u(s)|ds\\
	& \leq \|f-u\|_{L^p}\\
	&\leq C\|f\|_{L^p}.
\end{align*} 
And if we observe that the above remains valid for \(p=1\) and \(p=\infty\) we conclude that  $K_N(x)\leq C$ for all \(1\leq p\leq\infty\) and all \(0<x<1\).

Observe that from \cref{rem-p0-1,rem-p0-2,rem-p0-3} we know that for every \(1\leq p<\infty\) and every \(\alpha<1\) the solution \(u\) belongs to \(L^{p_0}\) for some \(p_0>p\) with \(\norm{u}_{L^{p_0}}\leq C\norm{f}_{L^p}\). From this fact we can write
\[
\abs{\int_0^xu(s)}ds\leq Cx^{1-\frac1{p_0}}\norm{f}_{L^p}.
\]

Now, for fixed \(0<x<1\) we define a function $f_x$ as
$$
f_x(t)=\begin{dcases}
	x^{-\frac{1}{p}}&  \text{if } 0\leq t\leq x,\\
	 0& \text{if } x<t\leq 1,
	 \end{dcases}
$$
which satisfies $\|f_x\|_{L^p}=1$ and
$$
\int_0^xf_x(s)ds=x^{1-\frac1p},
$$
for all \(1\leq p<\infty\). Therefore, for each $0<x<1$ we have
\begin{align*}
	K_N(x)&\geq \left|x^{\frac{1}{p} -1}\int_0^x (u(s)- f_x(s))ds\right|\\
	&\geq x^{\frac{1}{p} -1}\int_0^xf_x(s)ds -  x^{\frac{1}{p} -1}\int_0^x|u(s)|ds \\
	&\geq 1-Cx^{\frac{1}{p}-\frac1{p_0}}\norm{f}_{L^p},
\end{align*}
thus
\[
K_N(x)\geq 1-Cx^{\frac{1}{p}-\frac1{p_0}}\geq\frac12
\]
for all sufficiently small \(x\).

If \(p=\infty\), we take \(f\equiv 1\) and by \cref{obsbess} the solution \(u\) can be written as
\[
u(x)=A\phi_-(x)+1,
\]
for some \(A\neq 0\) so that
\[
x^{2\alpha-1}u'(x)=A(2-2\alpha)b_2+O(x^{2-2\alpha}),
\]
where \(O(x^\beta)\) is a quantity that can be bounded above and below by \(Cx^\beta\) for some constant \(C>0\). Thus 
\[
K_N(x)\geq \abs{A(2-2\alpha)b_2}-O(x^{2-2\alpha})\geq\frac{A(2-2\alpha)b_2}2
\]
for all sufficiently small \(x\).

\appendix

\section{Bessel's equations}\label{app-bessel}

Let us recall some results regarding Bessel's differential equations which are closely related to \eqref{SL-eq}. For $\nu\geq0$ we consider the modified Bessel equation 
\begin{equation}\label{bess1}
	y^2 f''(y)+yf'(y)-(y^2+\nu^2)f(y)=0,
\end{equation}

It is known that the solutions to \eqref{bess1} are given by the modified Bessel's functions. This fact and some additional properties of the modified Bessel functions are summarized in the following lemma (see \cite{W44} for a detailed treatise on Bessel functions)

\begin{lemma}\label{lemabesel2}
	For non-integer $\nu>0$, the general solution of \eqref{bess1} can be written as
	$$
	f_{\nu}(y)= C_1 I_{\nu}(y) + C_2 I_{-\nu}(y).
	$$
	The function $I_{\nu}$ is the modified Bessel functions of the first type, and it has the following power series expansion near the origin
	\begin{equation}\label{seriesI}
	I_{\nu}(y)=\sum_{m=0}^{\infty}\frac{1}{m! \, \Gamma(m+\nu +1)}\left(\frac{y}{2}\right)^{2m+\nu},
	\end{equation}
	and it satisfies
	\begin{equation}\label{der-I}
		I_\nu'(y)=\frac{I_{\nu+1}(y)+I_{\nu-1}(y)}{2}
	\end{equation}
	For integer \(\nu\geq 0\), the general solution of \eqref{bess1} can be written as
	$$
	f_{\nu}(y)= C_1 I_{\nu}(y) + C_2 K_{\nu}(y).
	$$
	The function $K_{\nu}$ is the modified Bessel function of the second which satisfies
	\begin{multline}\label{seriesK1}
	K_\nu(y)=\frac{1}{2} \sum_{m=0}^{\nu-1}
	\frac{(-1)^{m}(\nu-m-1)!}{m!}\left(\frac{y}{2}\right)^{2m-\nu}\\
	+(-1)^{{\nu}+1} \sum_{m=0}^{\infty}
	\frac{\left(\frac{y}{2}\right)^{\nu+2 m}}{m!(\nu+m)!}\left\{\log \left(\frac{y}{2} \right)-\frac{1}{2} \psi(m+1)-\frac{1}{2} \psi(\nu+m+1)\right\},
	\end{multline}
	for \(\nu>0\) and
	\begin{equation}\label{seriesK0}
	K_0(y)=-\log\pt{\frac{y}{2}}I_0(y)+\sum_{m=0}^{\infty}
		\frac{\left(\frac{y}{2}\right)^{2m}}{(m!)^2}\psi(m+1),
	\end{equation}
	where \(\psi(z)=\dfrac{\Gamma'(z)}{\Gamma(z)}\) is the digamma function. The function \(K_\nu\) satisfies in addition
	\begin{equation}\label{der-K}
		K_\nu'(y)=-\frac{K_{\nu+1}(y)+K_{\nu-1}(y)}{2}
	\end{equation}
\end{lemma}

\begin{remark}\label{rem-Bes1}
	Observe that for \(\nu>0\) the expansion \eqref{seriesK1} tells us that for \(y\sim 0\)
	\[
	K_{\nu}(y)=\frac{1}{2} \sum_{m=0}^{\nu-1}
	\frac{(-1)^{m}(\nu-m-1)!}{m!}\left(\frac{y}{2}\right)^{2m-\nu}+O\pt{y^\nu\log y},
	\]
	and that \eqref{seriesI} and \eqref{seriesK0} tell us that
	\[
	K_{0}(y)=-\gamma-\log\pt{\frac{y}{2}}+O(y^2\log y),
	\]
	where \(\gamma\) is Euler's constant.
\end{remark}

The above equations are relevant to us because of the following results, which is an adaptation of \cite[Lemma 4.1]{CW11-1}, but we will include its proof for the reader's convenience.

\begin{lemma} \label{lemabesel1}
	For $\alpha<1$, let $f_{\nu}$ a solution of \eqref{bess1} with $\nu=\abs{\frac{\frac12-\alpha}{1-\alpha}}$. Then $u(x)$, defined as $u(x)=x^{\frac{1}{2}-\alpha}f_{\nu}\left(\frac{x^{1-\alpha}}{1-\alpha}\right)$, solves
	$$-(x^{2\alpha}u'(x))'+u(x)=0.$$
\end{lemma}

\begin{proof}
Using \eqref{bess1} with $y=\frac{x^{1-\alpha}}{1-\alpha}$ we have
\begin{equation*}
	\frac{x^{2-2\alpha}}{(1-\alpha)^2}f''_{\nu}\left(y\right)+\frac{x^{1-\alpha}}{1-\alpha} f'_{\nu}\left(y\right)=\left[\frac{x^{2-2\alpha}}{(1-\alpha)^2} + \left(\frac{\frac12-\alpha}{1-\alpha}\right)^2\right] f_{\nu}(y) .
\end{equation*}
Multiplying by $(1-\alpha)^2x^{\alpha-\frac{3}{2}}$ we have
\begin{equation*}
	{x^{\frac{1}{2}-\alpha}}f''_{\nu}\left(y\right)+(1-\alpha){x^{-\frac{1}{2}}} f'_{\nu}\left(y\right)=\left[{x^{\frac{1}{2}-\alpha}} + \left({\frac12-\alpha}\right)^2x^{\alpha-\frac{3}{2}}\right] f_{\nu}(y) . 
\end{equation*}
	Then
\begin{equation} \label{b1}
	{x^{\frac{1}{2}-\alpha}}f''_{\nu}\left(y\right)+(1-\alpha){x^{-\frac{1}{2}}} f'_{\nu}\left(y\right) - \left({\frac{1}{2}-\alpha}\right)^2x^{\alpha-\frac{3}{2}} f_{\nu}(y)=x^{\frac{1}{2}-\alpha}f_{\nu}(y).
\end{equation}
Notice that if 
\begin{equation}\label{b2}
	u(x)=x^{\frac{1}{2}-\alpha}f_{\nu}\left(\frac{x^{1-\alpha}}{1-\alpha}\right) ,
\end{equation}
then
\begin{align*} 
	u'(x)&= \left(\frac{1}{2}-\alpha\right)x^{-\alpha -\frac{1}{2}}f_{\nu}(y)+x^{-2\alpha +\frac{1}{2}}f'_{\nu}(y),
\end{align*}
\begin{equation*}
	x^{2\alpha}u'(x)=\left(\frac{1}{2}-\alpha\right)x^{\alpha -\frac{1}{2}}f_{\nu}(y)+x^{\frac{1}{2}}f'_{\nu}(y).
\end{equation*}
Deriving
\begin{align*}
	(x^{2\alpha}u'(x))'
	&=- \left(\frac{1}{2}-\alpha\right)^2x^{\alpha-\frac{3}{2}}f_{\nu}(y)+\left(\frac{1}{2}-\alpha\right)x^{ -\frac{1}{2}}f'_{\nu}(y) \\
	&\qquad +\frac{1}{2}x^{-\frac{1}{2}}f_{\nu}'(y)+ x^{\frac{1}{2}-\alpha}f''_{\nu}(y) \\
	& = - \left(\frac{1}{2}-\alpha\right)^2x^{\alpha-\frac{3}{2}}f_{\nu}(y)+\left(1-\alpha\right)x^{ -\frac{1}{2}}f'_{\nu}(y) + x^{\frac{1}{2}-\alpha}f''_{\nu}(y).
	\end{align*}
	Replacing with \eqref{b1} and \eqref{b2} we have $(x^{2\alpha}u'(x))'=u(x)$. 
\end{proof}

The following lemma summarizes a few facts regarding \eqref{bess1} (see for instance \cite{W44}).

\begin{remark}
	\label{obsbess}    
	For $f\in C(0,1)$, the equation $-(x^{2\alpha} u'(x))'+u(x)=f(x)$ has a general solution given by $u(x)=A \phi_+(x)+B \phi_-(x)+F(x)$ where $\phi_+(x)$ and $\phi_-(x)$ are linearly independent solutions of $-(x^{2\alpha} u'(x))'+u(x)=0$ and $F(x)$ is given by
	$$
	F(x)=\phi_+(x)\int_0^x f(s)\phi_-(s)ds - \phi_-(x)\int_0^x f(s)\phi_+(s)ds.
	$$
	
	If \(\alpha<\frac12\) then \(\nu=\frac{\frac12-\alpha}{1-\alpha}\) is non-integer so from \cref{lemabesel1,lemabesel2} we deduce that $\phi_\pm(x)= x^{\frac{1}{2}-\alpha}I_{\pm \nu}\left(\frac{x^{1-\alpha}}{1-\alpha}\right)$ and therefore, we can verify that
	\[
	\phi_+(x)=a_1x^{1-2\alpha}+a_2x^{3-4\alpha}+a_3x^{5-6\alpha}+\cdots
	\]
	and
	\[
	\phi_-(x)=b_1+b_2x^{2-2\alpha}+b_3x^{4-4\alpha}+b_4x^{6-6\alpha}+\cdots,
	\]
	from where one obtains
	\[
	\phi_+'(x)=a_1(1-2\alpha)x^{-2\alpha}+a_2(3-4\alpha)x^{2-4\alpha}+a_3(5-6\alpha)x^{4-6\alpha}+\cdots
	\]
	and
	\[
	\phi_-'(x)=b_2(2-2\alpha)x^{1-2\alpha}+b_3(4-4\alpha)x^{3-4\alpha}+b_4(6-6\alpha)x^{5-6\alpha}+\cdots,
	\]
	for certain non-zero constants $a_i,$ $b_i$. Therefore for \(\alpha<\frac12\) we have
	\begin{equation}\label{prop-bes1}
	\phi_+(0)=0,
	\qquad\text{and}\qquad
	\phi_-(0)=b_1.
	\end{equation}
	and for any \(\alpha<1\)
	\begin{equation}\label{prop-bes2}
		\lim_{x\to 0}x^{2\alpha}\phi_+'(x)=a_1(1-2\alpha),
		\qquad\text{and}\qquad
		\lim_{x\to 0}x^{2\alpha}\phi_-'(x)=0.
	\end{equation}
\end{remark}

The above results allow us to obtain the following

\begin{lemma}\label{lemmabes3}
	For any \(\alpha<\frac12\) there exists $g\in C^2(0,1] \cap C^0[0,1]$ such as
	\begin{equation}\label{Aux}
		\left\{
		\begin{aligned}
			-(x^{2\alpha}g'(x))'+g(x)&=0 && \text{in } (0,1),\\
			g(1)&= 0, \\
			\lim_{x\to 0^+}g(x)&= 1.
		\end{aligned}
		\right.
	\end{equation}
	Furthermore $\lim\limits_{x\to 0^+}x^{2\alpha}g'(x)$ exists and it is non zero.
\end{lemma}

\begin{proof}
	Observe that we can take $g(x)=A \phi_+(x) + B\phi_-(x)$ as in the previous remark. Indeed, we have that $\lim\limits_{x\to 0+}g(x)=Bb_1 $, therefore, if we take $B=\frac{1}{b_1}$, then $g(1)= A\phi_+(1) + \frac{1}{b_1} \phi_-(1)$. Taking $A=-\frac{\phi_-(1)}{b_1 \phi_+(1)}$, we have the boundary conditions verified. Therefore it is clear that
	\[
	\lim\limits_{x\to 0^+}x^{2\alpha}g'(x)=\lim\limits_{x\to 0^+}Ax^{2\alpha}\phi_+'(x)=Aa_1(1-2\alpha)\neq 0
	\]	
\end{proof}

\providecommand{\bysame}{\leavevmode\hbox to3em{\hrulefill}\thinspace}
\providecommand{\MR}{\relax\ifhmode\unskip\space\fi MR }
\providecommand{\MRhref}[2]{%
	\href{http://www.ams.org/mathscinet-getitem?mr=#1}{#2}
}
\providecommand{\href}[2]{#2}


\begin{thebibliography}{10}
	
	\bibitem{CR-O2013-1}
	Xavier Cabr{\'e} and Xavier Ros-Oton, \emph{Regularity of stable solutions up
		to dimension 7 in domains of double revolution}, Comm. Partial Differential
	Equations \textbf{38} (2013), no.~1, 135--154. \MR{3005549}
	
	\bibitem{CKN1984}
	Luis~A. Caffarelli, Robert~V. Kohn, and Louis Nirenberg, \emph{First order
		interpolation inequalities with weights}, Compositio Math. \textbf{53}
	(1984), no.~3, 259--275. \MR{768824 (86c:46028)}
	
	\bibitem{Cas2016-2}
	Hern\'an Castro, \emph{Hardy-{S}obolev-type inequalities with monomial
		weights}, Ann. Mat. Pura Appl. (4) \textbf{196} (2017), no.~2, 579--598.
	\MR{3624966}
	
	\bibitem{Cas2017-1}
	Hern{\'a}n Castro, \emph{The essential spectrum of a singular
		{S}turm-{L}iouville operator}, Mathematische Nachrichten \textbf{291} (2018),
	no.~4, 593--609.
	
	\bibitem{CW11-1}
	Hern{\'a}n Castro and Hui Wang, \emph{A singular {S}turm-{L}iouville equation
		under homogeneous boundary conditions}, J. Functional Analysis \textbf{261}
	(2011), no.~6, 1542--1590. \MR{2813481 (2012f:34056)}
	
	\bibitem{CW11-2}
	\bysame, \emph{A singular {S}turm-{L}iouville equation under non-homogeneous
		boundary conditions}, Differential Integral Equations \textbf{25} (2012),
	no.~1-2, 85--92. \MR{2906548 (2012m:34042)}
	
	\bibitem{Serrin1964}
	James Serrin, \emph{Local behavior of solutions of quasi-linear equations},
	Acta Math. \textbf{111} (1964), 247--302. \MR{170096}
	
	\bibitem{St00-1}
	Charles~A. Stuart, \emph{Buckling of a tapered elastica}, C. R. Acad. Sci.
	Paris S\'er. I Math. \textbf{331} (2000), no.~5, 417--421. \MR{1784925
		(2001h:74024)}
	
	\bibitem{St01}
	\bysame, \emph{Buckling of a heavy tapered rod}, J. Math. Pures Appl. (9)
	\textbf{80} (2001), no.~3, 281--337. \MR{1826347 (2002b:74023)}
	
	\bibitem{St02}
	\bysame, \emph{On the spectral theory of a tapered rod}, Proc. Roy. Soc.
	Edinburgh Sect. A \textbf{132} (2002), no.~3, 729--764. \MR{1912424
		(2003i:34035)}
	
	\bibitem{SV03}
	Charles~A. Stuart and Gr{\'e}gory Vuillaume, \emph{Buckling of a critically
		tapered rod: global bifurcation}, Proc. R. Soc. Lond. Ser. A Math. Phys. Eng.
	Sci. \textbf{459} (2003), no.~2036, 1863--1889. \MR{1993662 (2004e:47108)}
	
	\bibitem{SV04}
	\bysame, \emph{Buckling of a critically tapered rod: properties of some global
		branches of solutions}, Proc. R. Soc. Lond. Ser. A Math. Phys. Eng. Sci.
	\textbf{460} (2004), no.~2051, 3261--3282. \MR{2098717 (2005g:34046)}
	
	\bibitem{V03}
	Gr{\'e}gory Vuillaume, \emph{A singular nonlinear eigenvalue problem:
		Bifurcation in non-differentiable cases}, Ph.D. thesis, 2003.
	
	\bibitem{Vui03}
	\bysame, \emph{Study of the buckling of a tapered rod with the genus of a set},
	SIAM J. Math. Anal. \textbf{34} (2003), no.~5, 1128--1151 (electronic).
	\MR{2001662 (2004i:47131)}
	
	\bibitem{W44}
	George~N. Watson, \emph{A {T}reatise on the {T}heory of {B}essel {F}unctions},
	Cambridge University Press, Cambridge, England, 1944. \MR{0010746 (6,64a)}
	
	\bibitem{Zet2005}
	Anton Zettl, \emph{Sturm-{L}iouville theory}, Mathematical Surveys and
	Monographs, vol. 121, American Mathematical Society, Providence, RI, 2005.
	\MR{2170950}
	
\end{thebibliography}
\end{document}